\documentclass[preprint,12pt]{elsarticle}

\usepackage{amssymb}
\usepackage{amsthm}

\newtheorem{theorem}{Theorem}

\usepackage{lipsum}
\usepackage{bm}
\usepackage{amsmath}
\usepackage{amssymb}
\usepackage{comment}
\usepackage{amsfonts}
\usepackage{graphicx}
\usepackage{epstopdf}
\usepackage{algorithmic}
\usepackage[T1]{fontenc}
\usepackage{textcomp}
\usepackage{listings}
\usepackage{hyperref}
\usepackage{subcaption}
\usepackage{tikz}
\usepackage[utf8]{inputenc}
\usepackage{pgfplots}
\usepackage{pbox}

\definecolor{efie}{rgb}{1,0.64,0.4}
\definecolor{calefie}{rgb}{1,0,0}
\definecolor{badmfie}{rgb}{0,1,1}
\definecolor{mfie}{rgb}{0,0,1}
\definecolor{cfie}{rgb}{0,0,0}
\definecolor{reference}{rgb}{0,.5,0}

\newcommand{\efieplot}[1][3]{\addplot [efie, mark=triangle*, mark size=#1, mark options={solid,draw=black}]}
\newcommand{\calefieplot}[1][2]{\addplot [calefie, mark=*, mark size=#1, mark options={solid,draw=black}]}
\newcommand{\mfieplot}[1][2]{\addplot [mfie, mark=square*, mark size=#1, mark options={solid,draw=black}]}

\newcommand{\cfieplot}[1][3]{\addplot [cfie, mark=diamond*, mark size=#1, mark options={solid,draw=black}]}

\newcommand{\efiedesc}{orange triangles}
\newcommand{\calefiedesc}{red circles}
\newcommand{\mfiedesc}{blue squares}

\newcommand{\cfiedesc}{black diamonds}

\pgfplotsset{every tick label/.append style={font=\footnotesize}}
\pgfplotsset{label style={font=\footnotesize}}

\ifpdf
  \DeclareGraphicsExtensions{.eps,.pdf,.png,.jpg}
\else
  \DeclareGraphicsExtensions{.eps}
\fi

\graphicspath{{img/}}

\usepackage{amsopn}

\usepackage{cleveref}

\newcommand{\B}[1]{{\bm#1}}
\newcommand{\bE}{\B{e}}

\newcommand{\bnu}{\B{\nu}}
\newcommand{\inc}{^\textup{inc}}
\newcommand{\scat}{^\textup{scat}}
\newcommand{\tot}{^\textup{tot}}

\newcommand{\ii}{\mathrm{i}}
\newcommand{\scurl}[1][]{\textup{curl}\ifthenelse{\equal{#1}{}}{}{_#1}}
\newcommand{\vcurl}[1][]{\textup{\textbf{curl}}\ifthenelse{\equal{#1}{}}{}{_#1}}
\newcommand{\sdiv}[1][]{\textup{div}\ifthenelse{\equal{#1}{}}{}{_#1}}
\newcommand{\vdiv}[1][]{\textup{\textbf{div}}\ifthenelse{\equal{#1}{}}{}{_#1}}
\newcommand{\popE}{\mathcal{E}}
\newcommand{\popH}{\mathcal{H}}
\newcommand{\bopE}{\mathsf{E}}
\newcommand{\bopH}{\mathsf{H}}
\newcommand{\bopI}{\mathsf{Id}}
\newcommand{\bopA}{\mathsf{A}}
\newcommand{\bopR}{\mathsf{R}}
\newcommand{\interior}{^\textup{--}}
\newcommand{\exterior}{^\textup{+}}
\newcommand{\gt}{\B{\gamma}_\textup{t}}
\newcommand{\gN}{\B{\gamma}_\textup{N}}
\newcommand{\gn}{\B{\gamma}_{\bnu}}

\newcommand{\gtint}{\gt\interior}

\newcommand{\sdivspaceGamma}{\B{H}_\times^{-1/2}(\sdiv[\Gamma],\Gamma)}
\newcommand{\scurlspaceGamma}{\B{H}_\times^{-1/2}(\scurl[\Gamma],\Gamma)}
\newcommand{\average}[1]{\left\{#1\right\}_\Gamma}
\newcommand{\jump}[1]{\left[#1\right]_\Gamma}

\journal{CAMWA}

\begin{document}

\begin{frontmatter}

\title{Software frameworks for integral equations in electromagnetic scattering based on Calder\'on identities}

\author[ucl]{Matthew Scroggs}
\ead{matthew.scroggs.14@ucl.ac.uk}
\author[ucl]{Timo Betcke}
\ead{t.betcke@ucl.ac.uk}
\author[ucl]{Erik Burman}
\ead{e.burman@ucl.ac.uk}
\author[simpleware]{Wojciech \'{S}migaj}
\ead{w.smigaj@simpleware.com}
\author[pontificia]{Elwin van 't Wout}
\ead{e.wout@ing.puc.cl}

\address[ucl]{University College London, London, UK}
\address[simpleware]{Simpleware Ltd, Exeter, UK}
\address[pontificia]{Pontificia Universidad Cat\'olica de Chile, Santiago, Chile}

\begin{abstract}
In recent years there have been tremendous advances in the theoretical understanding of boundary integral equations for Maxwell problems. 
In particular, stable dual pairings of discretisation spaces have been developed that allow robust formulations of the preconditioned 
electric field, magnetic field and combined field integral equations. Within the BEM++ boundary element library we have developed 
implementations of these frameworks that allow an intuitive formulation of the typical Maxwell boundary integral formulations within a few 
lines of code. The basis for these developments is an efficient and robust implementation of Calder\'on identities together with a product 
algebra that hides and automates most technicalities involved in assembling Galerkin boundary integral equations. In this paper we 
demonstrate this framework and use it to derive very simple and robust software formulations of the standard preconditioned electric field, 
magnetic field and regularised combined field integral equations for Maxwell.
\end{abstract}




\end{frontmatter}


\section{Introduction}
\label{sec:intro}
The numerical simulation of electromagnetic wave scattering poses significant theoretical and computational challenges. Much effort in 
recent years has gone into the development of fast and robust boundary integral equation formulations to simulate a range of phenomena from 
the design and performance of antennas to radar scattering from large metallic objects.

While there have been a range of important theoretical advances in recent years for the development of robust preconditioned boundary 
integral formulations for Maxwell, the computational implementation remains a challenge. At University College London, as part of the BEM++ 
project \cite{bempp} we have developed a number of easy to use Python-based open-source tools to explore and solve Maxwell problems based on 
preconditioned electric field (EFIE), magnetic field (MFIE) and combined field (CFIE) integral equation formulations. In this paper, we give an overview of these 
developments, present simple example codes and discuss the underlying implementation. The goal is to make advanced integral equation solvers for Maxwell available to non-specialised users without requiring significant knowledge in the design and implementation of these methods.

In particular, we 
consider the following formulation of electromagnetic scattering from a perfectly conducting object. Let $\Omega\interior\subset\mathbb{R}^3$ be 
a bounded domain with boundary $\Gamma$ and denote by $\Omega\exterior = \mathbb{R}^3\backslash\overline{\Omega\interior}$ its complement. Let $\bnu$ be 
the exterior normal vector on $\Gamma$ pointing into $\Omega\exterior$ as shown in Figure \ref{fig:exteriorproblem}.

\begin{figure}
  \centering
  \includegraphics[scale=.8]{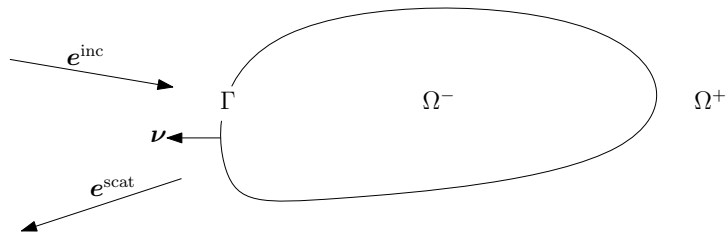}
  \caption{Time-harmonic scattering from a perfect conductor.}
  \label{fig:exteriorproblem}
\end{figure}

Denote by $\bE\inc$ an incident field. We are looking for the solution $\bE\tot=\bE\inc + \bE\scat$ of the exterior scattering problem, 
satisfying
\begin{subequations}
 \begin{align}
  \vcurl\,\vcurl\,\bE\tot-k^2\bE\tot &= 0&&\text{in }\Omega\exterior, \label{eq:maxwell_pde} \\
 \bE\tot\times \bnu&=0&&\text{on }\Gamma,\label{eq:maxwell_bnd}\\
\lim_{|\B{x}|\rightarrow\infty} |\B{x}|\left(\vcurl\,\bE\scat\times\frac{\B{x}}{|\B{x}|}-\ii k\bE\scat\right) &= 0&&\text{as }|\B{x}|\to\infty.\label{eq:maxwell_silver_mueller}
\end{align}
 \label{eq:maxwell}
\end{subequations}

\noindent Here, $k=\omega\sqrt{\epsilon_0\mu_0}$ denotes the wavenumber of the problem, with $\omega$ denoting the frequency and $\epsilon_0$ and $\mu_0$
the electric permeability and magnetic permittivity in vacuum. Frequently, the incident field is a plane wave given by $\bE\inc=\B{p}e^{\ii k 
\B{x}\cdot\B{d}}$, where $\B{p}$ is a non-zero vector representing the polarisation of the wave and $\B{d}$ is a unit vector perpendicular 
to $\B{p}$ that gives the direction of the plane wave.

In Section \ref{sec:bempp}, we give a short overview of the main concepts of the BEM++ boundary element library with an emphasis on its operator and grid function algebra.

When discretising these equations, the finite dimensional function spaces must be chosen carefully in order to give well-conditioned systems
of equations. In Section \ref{sec:spaces}, we present an overview of the function spaces for Maxwell boundary integral operators, and the various discretisations of these spaces. The emphasis is on the presentation of stable dual pairings between so-called Rao--Wilton--Glisson (RWG) \cite{RWG} spaces, and Buffa--Christiansen (BC) spaces \cite{BCSpacesDef}.

In Section \ref{sec:operators}, we present the standard electric and magnetic field boundary integral operators and the resulting Calder\'on projector. The concept of the Calder\'on projector is fundamental to this paper. We describe the function spaces used in assembling the Calder\'on projector and numerically demonstrate its properties using small BEM++ code snippets. 

In Sections \ref{sec:efie}, \ref{sec:mfie} and \ref{sec:cfie}, we then demonstrate the Calder\'on preconditioned EFIE, MFIE and regularised CFIE and their BEM++ implementations. Interesting non-trivial domains will be used to demonstrate and compare their properties.

The use of stable RWG/BC pairings, on which this paper is based, is well known in the electromagnetics community (see e.g. \cite{BCPrecond, AndriulliMFIE,JNM:JNM844}). The emphasis in this paper is on a simple high-level software representation that is mathematically correct and robust, but hides the underlying complexities to provide a simple framework in which to solve electromagnetic scattering problems.


\section{A Maxwell introduction to BEM++}
\label{sec:bempp}

BEM++ (\url{www.bempp.org}) is an open-source library for the Galerkin discretisation and solution of boundary integral equations in electrostatics, acoustics and computational electromagnetics.
The library supports dense discretisation and hierarchical matrix assembly. The interface of the library is written in Python with fast kernel assembly routines implemented in C++.

In this section we give a brief overview of BEM++ using the example of Maxwell boundary integral operators. We skip mathematical details: these will be discussed in the subsequent sections. A particular focus in this section is on the operator concept of BEM++, which will allow us to elegantly describe the various types of integral equation in what follows.

BEM++ has a range of features to define or import triangular surface grids consisting of flat elements. To define a function space over a grid we use the \verb@function_space@ command. For example, a simple space consisting of RWG functions (see Section \ref{sec:spaces} for details) is defined by
\begin{lstlisting}[language=Python]
import bempp.api
grid = ...
rwg_space = bempp.api.function_space(grid, 'RWG', 0)
\end{lstlisting}
The third parameter in the \verb@function_space@ command
is the degree of the space. For Maxwell only lowest order functions are supported.

With the definition of a function space we can now discretise functions over a given space into a \verb@GridFunction@ object. The following code computes a discretisation of the tangential trace of a plane wave Maxwell solution $\B{u}(\B{x}) = \B{p}e^{\ii k\B{d}\cdot\B{x}}$, where $\B{p}$ is the polarisation vector.
\begin{lstlisting}[language=Python]
import numpy as np
p = np.array([...])
d = np.array([...])
k = ...
def plane_wave(x):
    return np.exp(1j * np.dot(d, x)) * p

def tangential_trace(x, n, domain_index, result):
    result[:] = np.cross(plane_wave(x), n, axis=0)
    
grid_fun = bempp.api.GridFunction(
	space=rwg_space, fun=tangential_trace)    
\end{lstlisting}

The \verb@tangential_trace@ function defines the tangential
trace of the plane wave in dependence of a given
normal vector \verb@n@, which is passed in during the
discretisation phase. The \verb@GridFunction@ object
by default performs an $\B{L}^2$ projection onto the space \verb@rwg_space@. Oblique projections with a different test space are also supported by specifying the \verb@dual_space@ argument in the constructor of \verb@GridFunction@.

We now define an operator acting on grid functions.
An electric field integral operator is defined on the space $\sdivspaceGamma$ of surface-div-conforming functions. For the discretisation we need test functions from the space of surface-curl-conforming functions $\scurlspaceGamma$. Here, we choose RBC functions, which are explained in Section \ref{sec:spaces}.

\begin{lstlisting}[language=Python]
from bempp.api.operators.boundary import \
	maxwell
rbc_space = bempp.api.function_space(grid, 'RBC', 0)
op = maxwell.electric_field(
	rwg_space, rwg_space, rbc_space, k)
\end{lstlisting}
Each boundary operator in BEM++ takes three parameters, a  domain space, a range space, and a dual space to the range space. For a Galerkin discretisation the range space is not required. However, we use it to implement a Galerkin based operator algebra. This algebra allows us for example to write
\begin{lstlisting}[language=Python]
result = op * grid_fun
\end{lstlisting}
The result is a grid function, which is defined over the range space of the operator. Internally, this is implemented by multiplying the weak form of the operator with the coefficient vector of the grid function, and then solving with the mass matrix between the dual space and the range space to map the result back into a coefficient vector of the range space. This mass matrix solve in practice is only performed if a method requires the coefficient vector of the \verb@result@ grid function, so that the above code snippet is even well-defined if there is no stable mapping from the dual space to the range space. However, then only a representation of the grid function in the dual space is available.

By extending this mechanism to operators, the concept of products of operators is also defined in BEM++. Given two operators \verb@op1@ and 
\verb@op2@ in BEM++ with compatible spaces, their product is a new operator that is implicitly defined through first the 
application of \verb@op2@, followed by a mass matrix solve to map the result into the domain space of \verb@op1@, followed by the 
application of \verb@op1@.

The high-level concept of operators and grid functions extends everywhere into the library. For example, to solve the indirect electric field integral equation (EFIE) with the above plane wave incident field we can write
\begin{lstlisting}[language=Python]
sol = bempp.api.linalg.lu(op, grid_fun)
\end{lstlisting}
to solve the associated EFIE problem using a dense LU
decomposition. The object \verb@sol@ is a grid function
that lives in the domain space of \verb@op@ and is a grid
function that satisfies \verb@grid_fun == op * sol@ up
to numerical errors. Naturally, for larger problems we are interested in iterative solvers and corresponding interfaces to GMRES and CG are also provided.

We will use this operator concept extensively in this paper to formulate the different types of Maxwell integral equations. More details of 
the abstract implementation of operator algebras in BEM++ can be found in \cite{Betcke2017}.

\section{Function spaces}
\label{sec:spaces}
In this section we review the necessary function space definitions for boundary integral formulations of the electromagnetic scattering problem \eqref{eq:maxwell}.
In what follows, we assume that $\Gamma$ consists of a finite number of smooth faces that meet at non-degenerate edges.
This is a reasonable assumption from the point of view of boundary element methods, as we will consider discrete problems on meshed surfaces.

The fundamental difficulty lies in the correct description of trace spaces on $\Gamma$ for solutions of \eqref{eq:maxwell} \cite{Tartar97, 
Buffa2001, Buffa2001a, Buffa2002,BuffaOverview}. In this section, we summarise without proof some of the results of these papers, as they 
form the foundations for the operator definitions and stable implementations presented in later sections. Throughout, we will present small 
code snippets to demonstrate how to instantiate the corresponding spaces in BEM++. More details and references can also be found in the 
overview paper \cite{BuffaHiptmair} and in \cite{MeuryThesis}.

Throughout this paper, we adopt the convention of using bold and non-bold letters for spaces of vector- and scalar-valued functions, respectively. Whenever we write the subscript \text{loc} it is understood to only be used for the unbounded domain $\Omega^{+}$.
For definitions of Sobolev spaces of scalar-valued functions, we refer to \cite[chapter 3]{McLean}.

Let \textbf{op} be one of $\vcurl$, $\vcurl^2$ or $\sdiv$.
Denote by $\B{H}_\text{loc}(\textbf{op},\Omega^\pm)$ the function space defined as $\B{H}_\text{loc}(\textbf{op},\Omega^{\pm}):=\{\B{u}\in \B{L}^2_\text{loc}(\Omega^{\pm}):\textbf{op}\,\B{u}\in \B{L}^2_\text{loc}(\Omega^{\pm})\}$
(for $\sdiv$, the final $\B{L}^2_\text{loc}(\Omega^{\pm})$ should be replaced by the scalar space $L^2_\text{loc}(\Omega^{\pm})$).
To define the necessary function spaces on the surface $\Gamma$, we first define the tangential ($_\text{t}$) and Neumann ($_\text{N}$)
traces on $\Gamma$. These are defined, for $\B{p}\in \B{H}_\text{loc}(\vcurl, \Omega^\pm)$ and
$\B{q}\in \B{H}_\text{loc}(\vcurl^2,\Omega^\pm)$, by
\begin{align}
\gt^\pm\B{p}(\B{x})&:=\lim_{\Omega^\pm\ni\B{x}'\to\B{x}\in\Gamma}\B{p}(\B{x'})\times\bnu(\B{x}),&
\gN^\pm\B{q}(\B{x})&:=\frac{1}{\ii k}\gt^\pm\vcurl\,\B{q}(\B{x}),\label{eq:def_traces}
\end{align}
where the superscripts $\interior$ and $\exterior$ denote the interior and exterior traces, respectively. Note that in our definition $\gN^\pm$ contains an additional factor of $\ii$, which does not appear in \cite{BuffaHiptmair}. The interpretation is that if we normalise the magnetic permittivity and electric permeability to $1$, this definition of $\gN^\pm$ is the tangential trace of the magnetic field data.

In what follows we need the average 
$\average{\cdot}$, and jump, $\jump{\cdot}$ of these traces, defined as \begin{align}
\average{\B{\gamma}_*}\B{f}&:=\frac12(\B{\gamma}_*\exterior \B{f}+\B{\gamma}_*\interior \B{f}),&
\jump{\B{\gamma}_*}\B{f}&:=\B{\gamma}_*\exterior \B{f}-\B{\gamma}_*\interior \B{f}.
\end{align}
Let
$\B{L}^2_\textup{t}(\Gamma)$, defined by
\begin{equation}
\B{L}^2_\textup{t}(\Gamma):=\{\B{u}\in\B{L}^2(\Gamma):\B{u}\cdot\bnu=0\},
\end{equation}
be the space of square integrable tangential functions.
We define the tangential trace space, $\B{H}_\times^{1/2}(\Gamma)$, as in \cite[definition 1]{BuffaHiptmair} by
\begin{equation}\B{H}_\times^{1/2}(\Gamma):=\gtint(\B{H}^1(\Omega\interior))=\left\{\gtint\B{u}:\B{u}\in\B{H}^1(\Omega\interior)\right\}.\end{equation}
The dual of this space with respect to the antisymmetric product,
\begin{align}
\left\langle\B{a},\B{b}\right\rangle_{\tau}&:=
\int_\Gamma\B{a}\cdot(\bnu\times\B{b}),&\text{ for }\B{a},\B{b}\in\B{L}^2_\textup{t}(\Gamma).\label{eq:antisymmetric}
\end{align}
is denoted by $\B{H}_\times^{-1/2}(\Gamma)$.

Due to the assumption that $\Gamma$ consists of a finite number of smooth faces, we may let
$\Gamma=\bigcup_{j=0}^m\Gamma^j$, where
$\Gamma_0$, ..., $\Gamma_m$ are smooth. 
We define the scalar surface divergence of $\gt^\pm\B{u}$, for $\B{u}\in\B{H}_\text{loc}(\vcurl,\Omega^\pm)$, by
\begin{equation}
\sdiv[\Gamma](\gt^\pm\B{u}):=\gn^\pm(\vcurl\,\B{u}),
\end{equation}
where the normal trace, $\gn^\pm$, is defined for $\B{p}\in\B{H}_\text{loc}(\vdiv,\Omega^\pm)$ by
\begin{equation}
\gn^\pm\B{p}:=\lim_{\Omega^\pm\ni\B{x}'\to\B{x}\in\Gamma}\B{p}(\B{x'})\cdot\bnu(\B{x}).
\end{equation}
For a function $\B{u}\in\gt\interior(\B{C}^\infty(\overline{\Omega\interior}))$, this can be calculated using
\begin{equation}\sdiv[\Gamma]\B{u}:=\begin{cases}
\sdiv[j]\B{u}^j &\text{on }\Gamma^j\\
(\B{u}^j\cdot\B{\nu}^{ij}+\B{u}^i\cdot\B{\nu}^{ji})\delta_{ij}&\text{on }\overline{\Gamma^j}\cap\overline{\Gamma^i}
\end{cases},\end{equation}
where $\B{u}^j$ is the restriction of $\B{u}$ to the face $\Gamma^j$,
$\B{\nu}^{ij}$ is the outward pointing tangential normal to $\Gamma^i$ restricted to the edge $\overline{\Gamma^i}\cap\overline{\Gamma^j}$,
$\sdiv[j]$ is the two dimensional divergence computed on the face $\Gamma^i$,
and $\delta_{ij}$ is the Dirac delta distribution with support on the edge $\overline{\Gamma^i}\cap\overline{\Gamma^j}$.
By density, this definition can be extended to $\B{u}\in \B{H}_\times^{-1/2}(\Gamma)$.
We now define the space of surface-div-conforming functions by
\begin{equation}
\sdivspaceGamma:=\{\B\mu\in \B{H}_\times^{-1/2}(\Gamma):\sdiv[\Gamma]\B\mu\in H^{-1/2}(\Gamma)\}.
\end{equation}

The scalar surface curl may be defined, for $\B{u}\in\B{H}^{-1/2}_\times(\Gamma)$, by
\begin{equation}\scurl[\Gamma](\B{u}):=\sdiv[\Gamma](\B{u}\times\B\nu),\end{equation}
and
the space of surface-curl-conforming functions by
\begin{align}
\scurlspaceGamma:=\{\B\nu\times\B\mu:\B\mu\in\sdivspaceGamma\}.
\end{align}

By restricting their domains as in the following theorem, we obtain continuous and surjective traces.
\begin{theorem}
The traces 
\begin{align}
\gt^\pm&:\B{H}_\textup{loc}(\vcurl,\Omega^\pm)\to\sdivspaceGamma\\
\text{and }\gN^\pm&:\B{H}_\textup{loc}(\vcurl^2,\Omega^\pm)\to\sdivspaceGamma
\end{align}
are continuous and surjective.
\end{theorem}
\begin{proof}
See \cite[Theorem 4.1]{Buffa2002}
and \cite[Lemma 3]{BuffaHiptmair}.
\end{proof}

The antisymmetric dual form defined in \eqref{eq:antisymmetric} is intimately connected with the space $\sdivspaceGamma$. In \cite[Lemma 
5.6]{Buffa2002} it is shown that the space $\sdivspaceGamma$ is self-dual with respect to this antisymmetric dual form. Another 
interpretation of this dual form is as the standard $\B{L}^2$ dual between the spaces $\sdivspaceGamma$ and $\scurlspaceGamma$ since 
$\B\psi\in\scurlspaceGamma$ if and only if $\B{\psi}=\bnu\times\B{\xi}$ for some $\B{\xi}\in\sdivspaceGamma$.

Commonly used discretisations of these two spaces are Raviart-Thomas (RT) div-conforming \cite{RaviartThomas} and N\'ed\'elec (NC) curl-conforming \cite{Nedelec} basis functions.
\begin{figure}
\centering
\includegraphics{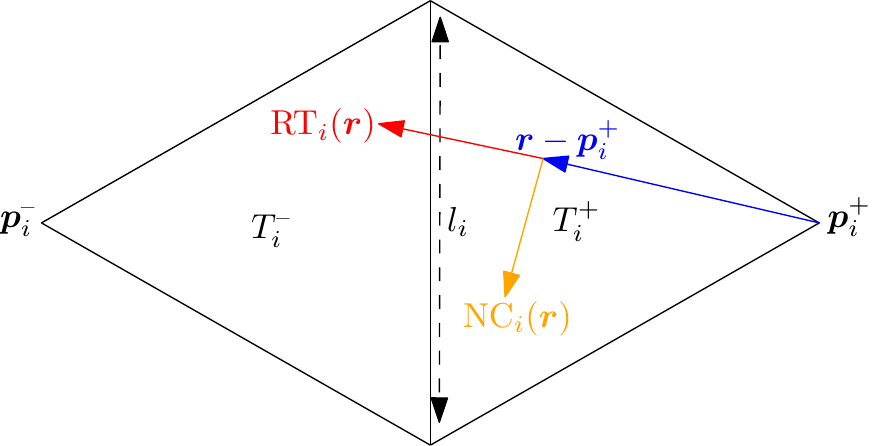}
\caption{Two adjacent triangles on which Raviart--Thomas (RT) and N\'ed\'elec (NC) basis functions are defined.}
\label{fig:rtnc_def}
\end{figure}
For the $i$th edge in a mesh, between two triangles $T_i^\text{+}$ and
$T_i^\text{--}$, the order 0 RT basis function is defined by
\begin{align}
\mathbf{RT}_i(\B{r}):=\begin{cases}
\frac{1}{2A_i^\text{+}}(\mathbf{r}-\mathbf{p}_i^\text{+})\quad&\text{if }\mathbf{r}\in T_i^\text{+}\\
-\frac{1}{2A_i^\text{--}}(\mathbf{r}-\mathbf{p}_i^\text{--})\quad&\text{if }\mathbf{r}\in T_i^\text{--}\\
\mathbf{0}\quad&\text{otherwise}
\end{cases},
\end{align}
where $A_i^\text{+}$ and $A_i^\text{--}$ are the areas of $T_i^\text{+}$ and $T_i^\text{--}$,
and $\B{p}_i^\text{--}$ and $\B{p}_i^\text{+}$ are the corners of $T_i^\text{+}$ and $T_i^\text{--}$ not on the shared edge,
as shown in Figure \ref{fig:rtnc_def}.
For the same edge, the order 0 NC basis function may be defined by
\begin{align}
\mathbf{NC}_i(\B{r}):=\bnu\times\mathbf{RT}_i(\B{r}).\label{eq:def_nc}
\end{align}
Example RT and NC order 0 basis functions are shown in Figure \ref{fig:rt_nc}.
In BEM++, RT and NC spaces may be created with the following lines of Python.
\begin{lstlisting}[language=Python]
rt_space = bempp.api.function_space(grid, "RT", 0)
nc_space = bempp.api.function_space(grid, "NC", 0)
\end{lstlisting}

\begin{figure}
\centering
\includegraphics[width=.39\textwidth]{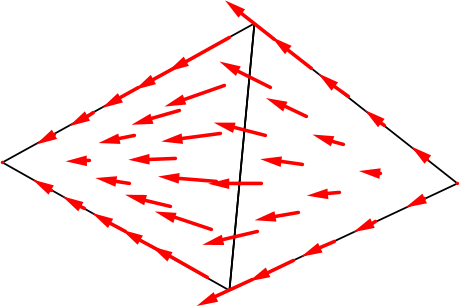}
\hspace{3mm}
\includegraphics[width=.39\textwidth]{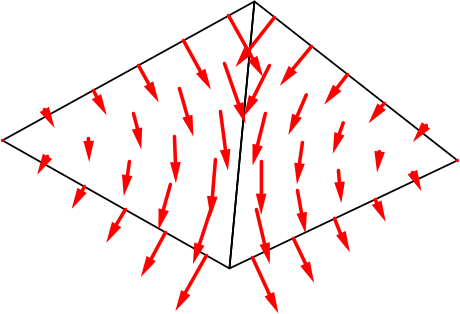}
\caption{A div-conforming Raviart--Thomas (left) and a curl-conforming N\'ed\'elec (right) order 0 basis function.}
\label{fig:rt_nc}
\end{figure}

The RT basis functions are closely related to the Rao--Wilton--Glisson (RWG) basis functions presented in \cite{RWG}. These are defined by
\begin{align}
\mathbf{RWG}_i(\B{r}):=l_i\mathbf{RT}_i(\B{r})=\begin{cases}
\frac{l_i}{2A_i^\text{+}}(\mathbf{r}-\mathbf{p}_i^\text{+})\quad&\text{if }\mathbf{r}\in T_i^\text{+}\\
-\frac{l_i}{2A_i^\text{--}}(\mathbf{r}-\mathbf{p}_i^\text{--})\quad&\text{if }\mathbf{r}\in T_i^\text{--}\\
\mathbf{0}\quad&\text{otherwise}
\end{cases},
\end{align}
where $l_i$ is the length of the shared edge, and all other terms are as above. We define the scaled curl-conforming dual basis functions of the RWG functions as
\begin{align}
\mathbf{SNC}_i(\B{r}):=l_i\mathbf{NC}_i(\B{r})=\bnu\times\mathbf{RWG}_i(\B{r}).
\end{align}
In BEM++, RWC and SNC spaces may be created with the following lines of Python.
\begin{lstlisting}[language=Python]
rwg_space = bempp.api.function_space(grid, "RWG", 0)
snc_space = bempp.api.function_space(grid, "SNC", 0)
\end{lstlisting}

An overview of other bases that can be used to discretise $\sdivspaceGamma$ and $\scurlspaceGamma$,
as well as other spaces, can be found in \cite{FenicsBook4}.

\subsection{Buffa--Christiansen spaces}
\label{sec:bc_spaces}
The RT basis functions have a subspace that is quasi-orthogonal to the space of curl-conforming N\'ed\'elec functions
\cite[section 3.1]{ChristiansenNedelec}.
Due to this, the antisymmetric bilinear form, as defined in \eqref{eq:antisymmetric}, on the discrete RT space is not inf-sup stable.
The motivation for Buffa--Christiansen (BC) basis functions is to find a space of functions that are div-conforming but behave like curl-conforming 
functions, as this will recover inf-sup stability.

To define such a basis, we first take the space of div-conforming RWG functions on a barycentric refinement of the original grid.
Given the $i$th edge of the original coarse grid, the associated BC function $\mathbf{BC}_i$ is now defined as a linear 
combination of these RWG functions on the barycentric refinement such that this linear combination is approximately tangentially (and therefore 
curl-) conforming across the $i$th edge. A typical BC basis function is shown in Figure \ref{fig:bc}. The full definition of BC functions is 
given in \cite{BCSpacesDef} in which also the following inf-sup stability result is shown. This result implies the stability of the Gram matrix 
between BC functions and RWG functions with respect to the dual pairing in \eqref{eq:antisymmetric}.

\begin{equation}
\inf_{\B{u}\in\text{RWG}}\sup_{\B{v}\in\text{BC}}\frac{\langle\B{u},\B{v}\rangle_\tau}{\|\B{u}\|_{-\tfrac12,\sdiv}\|\B{v}\|_{-\tfrac12,\sdiv}}>\frac1C.
\end{equation}

\begin{figure}
\centering
\includegraphics[width=.9\textwidth]{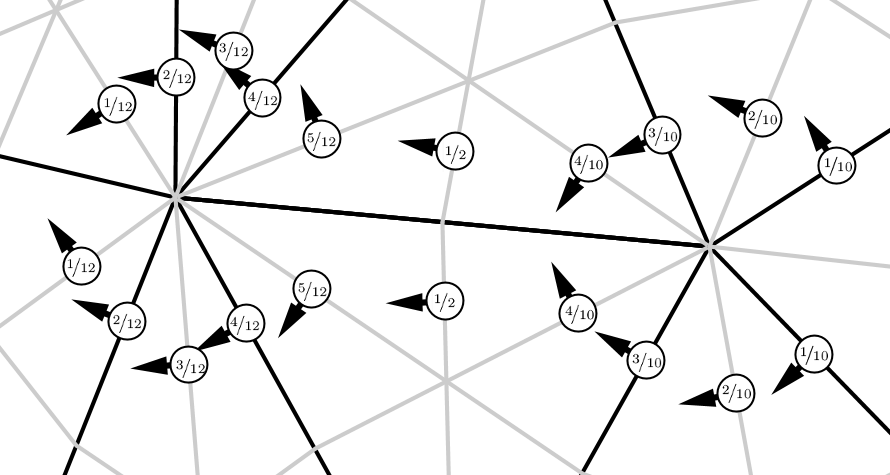}
\caption{The coefficients used to define a BC basis function in terms of RWG functions on the barycentrically refined grid.}
\label{fig:bc-def}
\end{figure}

\begin{figure}
\centering
\includegraphics[width=.8\textwidth]{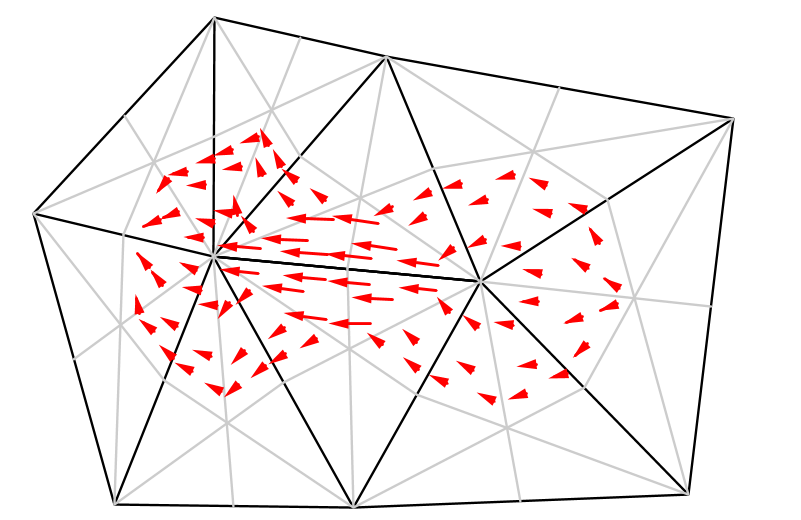}
\caption{A div-conforming and quasicurl-conforming Buffa--Christansen basis function, defined using RWG functions on barycentrically refined triangles.}
\label{fig:bc}
\end{figure}

For a BC basis function, $\mathbf{BC}_i$, we may also define the rotated Buffa--Christiansen (RBC) basis function, in an analogous way to
\eqref{eq:def_nc}, by
\begin{align}
\mathbf{RBC}_i(\B{r}):=\bnu\times\mathbf{BC}_i(\B{r}).
\end{align}
In BEM++, BC and RBC spaces may be created with the following lines of Python.
\begin{lstlisting}[language=Python]
bc_space = bempp.api.function_space(grid, "BC", 0)
rbc_space = bempp.api.function_space(grid, "RBC", 0)
\end{lstlisting}

\begin{figure}
\centering
\begin{lstlisting}[language=Python]
import numpy as np
import bempp.api
grid = bempp.api.shapes.regular_sphere(3)
rwg_space = bempp.api.function_space(grid, "B-RWG", 0)
rbc_space = bempp.api.function_space(grid, "RBC", 0)
snc_space = bempp.api.function_space(grid, "B-SNC", 0)
from bempp.api.operators.boundary.sparse \
	import identity as ident
id1 = ident(rwg_space, rwg_space, snc_space).weak_form()
id2 = ident(rwg_space, rwg_space, rbc_space).weak_form()
print(np.linalg.cond(
		bempp.api.as_matrix(id1).todense()))
print(np.linalg.cond(
		bempp.api.as_matrix(id2).todense()))
\end{lstlisting}
\caption{Computing the condition number of the RWG/SNC and the RWG/RBC mass matrix in BEM++. The values computed are $7.7\times10^{17}$ and $3.60$.}
\label{fig:conditioning_code}
\end{figure}

We can use BEM++ to compare the stability of dual pairings of RWG spaces with SNC and RBC spaces. The code in Figure \ref{fig:conditioning_code} computes the condition number of
mass matrices on a regular sphere grid generated
from the RWG/SNC pairing (\verb@id1@) and the RWG/RBC pairing (\verb@id2@).

For the condition number of \verb@id1@ the code computes
a value of $7.7\times 10^{17}$ and for the condition number of
\verb@id2@ it computes a value of $3.60$. We note that
in the definitions of the spaces in Figure \ref{fig:conditioning_code} we have used
the identifiers \verb@B-RWG@ and \verb@B-SNC@ instead of
\verb@RWG@ and \verb@SNC@. The reason is that the
RBC spaces are defined over barycentric refinements of
the grid. So we need to tell also the other space
definitions to internally use barycentric refinements
of the grid (even though the actual spaces live on the coarse grid), which is done by prepending \verb@B-@ in
the definitions.

\section{Operators}
\label{sec:operators}
BEM++ provides the magnetic and electric field domain potential and boundary operators.

First, we define the electric and magnetic potential operators (see \cite{BuffaHiptmair}),
$\popE,\popH:\sdivspaceGamma\to\B{H}_\text{loc}(\vcurl^2,\Omega\exterior\cup\Omega\interior)$,
by
\begin{align}
\popE(\B{p})(\B{x})&:=\ii k\int_\Gamma\B{p}(\B{y})G(\B{x},\B{y})\,d\B{y}
-\frac1{\ii k}\nabla_{\B{x}}\int_\Gamma\sdiv[\Gamma]\B{p}(\B{y})G(\B{x},\B{y})\,d\B{y},\\
\popH(\B{p})(\B{x})&:=
\vcurl[\B{x}]\int_\Gamma\B{p}(\B{y})G(\B{x},\B{y})\,d\B{y},
\end{align}
where $\displaystyle G(\B{x},\B{y})=\frac{e^{\ii k|\B{x}-\B{y}|}}{4\pi|\B{x}-\B{y}|}$ is the Green's function of a three-dimensional Helmholtz problem.

The definition of the electric potential operator, $\popE$, used here differs from that used in \cite{BuffaHiptmair} by a factor of $\ii$, corresponding to the modified definition of $\gN^{\pm}$.

With the electric and magnetic field operators we obtain the following representation formula.
\begin{theorem}
If $\bE\in\B{H}_\textup{loc}(\vcurl^2,\Omega\interior\cup\Omega\exterior)$ is a solution of \eqref{eq:maxwell_pde}, then 
\begin{align}
\bE(\B{x}) = -\popH([\gt]_\Gamma\bE)(\B{x})-\popE([\gN]_\Gamma\bE)(\B{x})\label{eq:general_rep}.
\end{align}
\end{theorem}
\begin{proof}
See \cite[section 4]{BuffaHiptmair}.
\end{proof}
Once the jumps of the traces of the solution are known or estimated on $\Gamma$, the representation formula \eqref{eq:general_rep}
can be used to find the solution at points in $\Omega^\pm$.
In BEM++ the electric and magnetic domain potential operators are available to evaluate this representation formula.
\begin{lstlisting}[language=Python]
from bempp.api.operators.potential import maxwell
eval_points = ...
e_pot = maxwell.electric_field(
	magnetic_space, eval_points, k)
m_pot = maxwell.magnetic_field(
	electric_space, eval_points, k)
field = - e_pot * magnetic_jump - m_pot * electric_jump
\end{lstlisting}
The domain is represented by a $3\times N$ array of $N$
points away from the boundary $\Gamma$. By default the potential operators are assembled using $\mathcal{H}$-matrices and then for given trace data evaluated by a fast $\mathcal{H}$-matrix/vector product \cite{MR3445676}.

Taking traces of the electric and magnetic field potential operators we arrive at the electric boundary operator, $\bopE:\sdivspaceGamma\to\sdivspaceGamma$, and the magnetic boundary operator, $\bopH:\sdivspaceGamma\to\sdivspaceGamma$. These are defined by
\begin{align}
\bopE&:=\average{\gt}\popE,&
\bopH&:=\average{\gt}\popH\label{eq:avgs}.
\end{align}
Additionally, we define the identity operator, $\bopI$, that maps every function to itself.
Because of the symmetry between electric and magnetic fields, the average Neumann traces can be written in terms of $\bopE$ and $\bopH$ as follows:
\begin{align}
\average{\gN}\popE&=\bopH,&
\average{\gN}\popH&=-\bopE.\label{eq:avg_neumann}
\end{align}
The following jump conditions can be derived \cite[Theorem 7]{BuffaHiptmair}.
\begin{align}
\jump{\gt}\popE&=\jump{\gN}\popH=0,&
\jump{\gN}\popE&=\jump{\gt}\popH=-\bopI.\label{eq:jumps}
\end{align}
Combining \eqref{eq:avgs},\eqref{eq:avg_neumann} and \eqref{eq:jumps} gives 
\begin{align}
\gt\exterior\popE&=\bopE,&
\gN\exterior\popE&=-\tfrac12\bopI+\bopH,\\
\gt\exterior\popH&=-\tfrac12\bopI+\bopH,&
\gN\exterior\popH&=-\bopE.
\end{align}
for the exterior traces and 
\begin{align}
\gt\interior\popE&=\bopE,&
\gN\interior\popE&=\tfrac12\bopI+\bopH,\\
\gt\interior\popH&=\tfrac12\bopI+\bopH,&
\gN\interior\popH&=-\bopE.
\end{align}
for the interior traces.

For given \verb@domain@, \verb@range_@ and \verb@dual_to_range@ spaces the electric and magnetic field boundary operators in BEM++ are defined by
\begin{lstlisting}[language=Python]
from bempp.api.operators.boundary import maxwell
efie = maxwell.electric_field(
	domain, range_, dual_to_range, k)
mfie = maxwell.magnetic_field(
	domain, range_, dual_to_range, k)
\end{lstlisting}

To obtain standard RWG discretisations of these operators we choose RWG as domain and range space, and SNC as dual to range space. It is 
important to note that internally we use the formulations of the weak forms of these operators given in \cite{BuffaHiptmair}, which are 
based on the antisymmetric dual form $\langle\cdot,\cdot\rangle_{\tau}$ and not on the standard $\B{L}^2$ dual form. Hence, the dual spaces are 
the non-rotated spaces RWG and BC for the discretisation, and when SNC is passed as dual to range space, it is internally interpreted as an 
RWG space for the generation of the discrete weak form. 
The rotated spaces SNC and RBC play a role in the operator algebra used to form the mass matrices for the $\langle\cdot, \cdot\rangle_{\tau}$ dual form by means of the standard $\B{L}^2$ dual form.



\subsection{The Calder\'on projector and its discretisation}
\label{sec:calderon}
Using the boundary operators in the previous section, we define the multitrace operator $\bopA$ by
\begin{equation}
\bopA:=\begin{bmatrix}
\bopH&\bopE\\
-\bopE&\bopH\\
\end{bmatrix}.
\end{equation}
We then define the exterior Calder\'on projector, $\mathcal{C}\exterior$, as follows.
\begin{equation}
\mathcal{C}\exterior := \tfrac12\bopI-\bopA = 
\begin{bmatrix}
\tfrac12\bopI-\bopH
&
-\bopE
\\
\bopE
&
\tfrac12\bopI-\bopH
\end{bmatrix}.
\end{equation}

It is well known \cite{BuffaHiptmair} that, given any arbitrary trace data
$(\B{a},\B{b})\in \sdivspaceGamma^2$
the product $$\mathcal{C}\exterior\begin{bmatrix}\B{a}\\\B{b}\end{bmatrix}
$$
defines a compatible pair of Cauchy data, and
$$
\left[\mathcal{C}\exterior\right]^2\begin{bmatrix}\B{a}\\\B{b}\end{bmatrix} = \mathcal{C}\exterior\begin{bmatrix}\B{a}\\\B{b}\end{bmatrix}
$$
from which we obtain $\left[\mathcal{C}\exterior\right]^2 = \mathcal{C}\exterior$. Using this identity and the representation $\mathcal{C}\exterior = \tfrac12\bopI-\bopA$ we obtain
\begin{equation}
\label{eq:calderon_square}
\bopA^2 = \tfrac14\bopI.
\end{equation}
This relationship is crucial for preconditioning numerical methods based on the Calder\'on projector and any discretisation scheme should preserve this property.

Denote by
\begin{equation}
A:=\begin{bmatrix}
H_1&E_1\\-E_2&H_2
\end{bmatrix},\label{def:discrete_multitrace}
\end{equation}
the discretisation of the operator $\bopA$. Here,
$E_1$ and $E_2$ are discretisations of electric field operators
and $H_1$ and $H_2$ are discretisations of magnetic field operators.

We want to numerically satisfy the relationship in \eqref{eq:calderon_square}. Hence, we require that
$$
\begin{bmatrix}
H_1&E_1\\-E_2&H_2
\end{bmatrix}
\begin{bmatrix}
M_1^{-1} & \\
 & M_2^{-1}
\end{bmatrix}
\begin{bmatrix}
H_1&E_1\\-E_2&H_2
\end{bmatrix}
\approx \frac{1}{4}
\begin{bmatrix}
M_1&\\&M_2
\end{bmatrix},
$$
where $M_1$ and $M_2$ are the corresponding mass matrices between the dual spaces and range spaces of the operators in the first line, and in the second line, respectively.

\begin{table}
\centering
\begin{tabular}{|c|l|c|c|c|}
\hline
Matrix&Operator&Domain&Range&Dual to Range\\
\hline
\hline
$H_1$&Magnetic&RWG&RWG&RBC\\
\hline
$E_1$&Electric&BC&RWG&RBC\\
\hline
$E_2$&Electric&RWG&BC&SNC\\
\hline
$H_2$&Magnetic&BC&BC&SNC\\
\hline
\end{tabular}
\caption{Spaces to use when discretising the multitrace operator, $\bopA$.}\label{table:multitrace_spaces}
\end{table}

In order to satisfy the above relationship it is crucial that $M_1$ and $M_2$ are well-conditioned mass matrices. A choice of spaces for the operators that achieves this goal, is shown in Table \ref{table:multitrace_spaces}. These choices of spaces lead to all mass matrices in the discretisation of $\bopA^2$ being the invertible RWG-RBC or BC-SNC pairings. The choice of spaces in Table \ref{table:multitrace_spaces} is based on representing the tangential trace with an RWG space and the Neumann trace with a BC space. Alternatively, one could use BC for the electric component and RWG for the magnetic component. This would lead to a discretisation in which $E_1$ and $E_2$ are swapped, and $H_1$ and $H_2$ are swapped.

Using the BEM++ library, the stable multitrace operator may be created using the following lines of Python.
\begin{lstlisting}[language=Python]
from bempp.api.operators.boundary import maxwell
multitrace = maxwell.multitrace_operator(grid, k)
\end{lstlisting}
We may then create the exterior Calder\'on projector with the following lines.
\begin{lstlisting}[language=Python]
from bempp.api.operators.boundary import sparse
identity = sparse.multitrace_identity(
    grid, spaces="maxwell")
calderon = 0.5 * identity - multitrace
\end{lstlisting}

\begin{figure}
\begin{lstlisting}[language=Python]
import bempp.api
from bempp.api.operators.boundary import maxwell
from bempp.api.operators.boundary import sparse
import numpy as np

k = 2

grid = bempp.api.shapes.cube(h=0.1)
multitrace = maxwell.multitrace_operator(grid, k)
identity = sparse.multitrace_identity(
    grid, spaces='maxwell')

calderon = 0.5 * identity - multitrace

def tangential_trace(x, n, domain_index, result):
    result[:] = np.cross(np.array([1, 0, 0]), n)

electric_trace = bempp.api.GridFunction(
    space=calderon.domain_spaces[0], 
    fun=tangential_trace,
    dual_space=calderon.dual_to_range_spaces[0])

magnetic_trace = bempp.api.GridFunction(
    space=calderon.domain_spaces[1], 
    fun=tangential_trace,
    dual_space=calderon.dual_to_range_spaces[1])

traces_1 = calderon * [electric_trace, magnetic_trace]
traces_2 = calderon * traces_1
electric_error = (traces_2[0] - traces_1[0]).l2_norm() \
	/ traces_1[0].l2_norm()
magnetic_error = (traces_2[1] - traces_1[1]).l2_norm() \
	/ traces_1[1].l2_norm()
\end{lstlisting}
\caption{Applying the Calder\'on projector to the tangential trace of the constant vector $(1, 0, 0)$ for the electric and magnetic trace, and computing the error in the magnetic and electric trace between the application of $\left[\mathcal{C}\exterior\right]^2$ and $\mathcal{C}\exterior$ to this trace data.}
\label{fig:calderon_code}
\end{figure}

A complete example for a computation with the Calder\'on projector is given in Figure \ref{fig:calderon_code}. It takes as electric and magnetic trace data the tangential trace of the constant vector $(1, 0, 0)$. This is obviously not a pair of compatible Cauchy data for Maxwell. It then applies the Calder\'on projector to obtain the numerically compatible pair of Cauchy data \verb@traces_1@ and then applies the Calder\'on projector one more time to obtain the pair of traces \verb@traces_2@. In a stable implementation, both, \verb@traces_1@ and \verb@traces_2@ should agree up to discretisation error, and indeed, the error \verb@electric_error@ in the electric component is $9.8\times 10^{-3}$ and the error \verb@magnetic_error@ in the magnetic component is $7.4\times 10^{-3}$.

It is important to note that we need to choose different discretisation spaces for the electric and magnetic trace. For the electric trace we choose RWG basis functions, and for the magnetic trace we choose BC basis functions.

\begin{figure}
\begin{subfigure}{\textwidth}
  \centering
  \includegraphics[scale=.15]{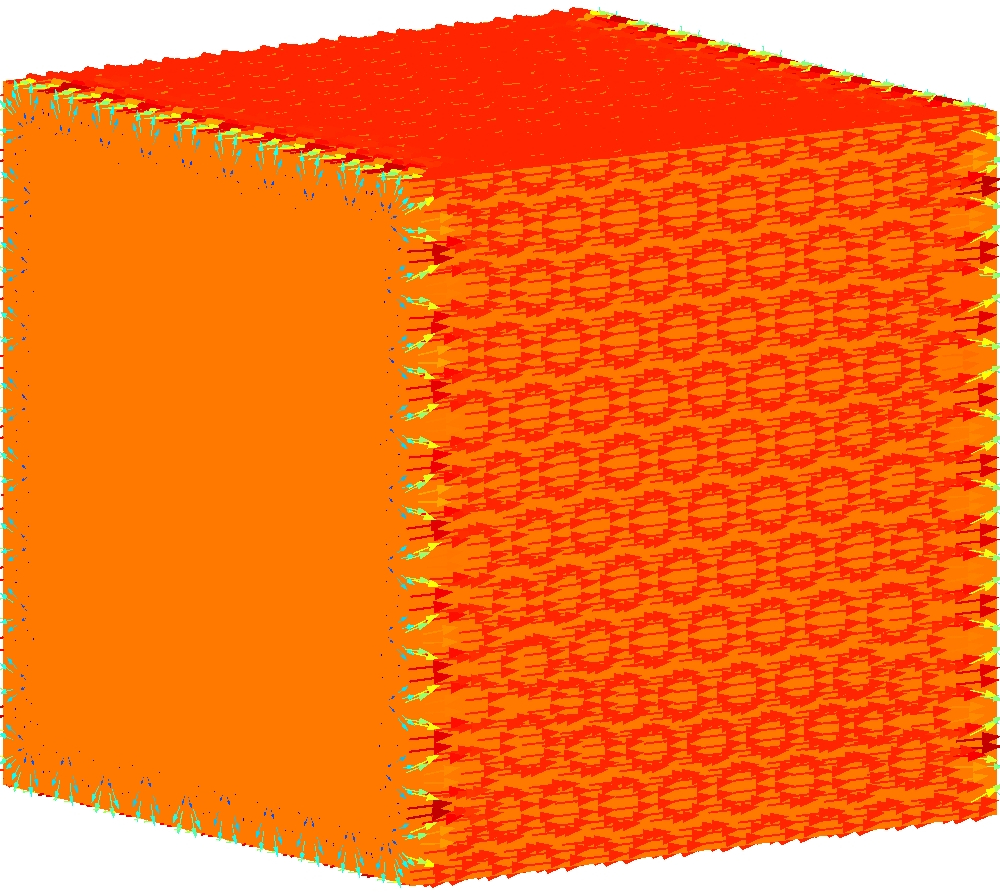}
  \includegraphics[scale=.15]{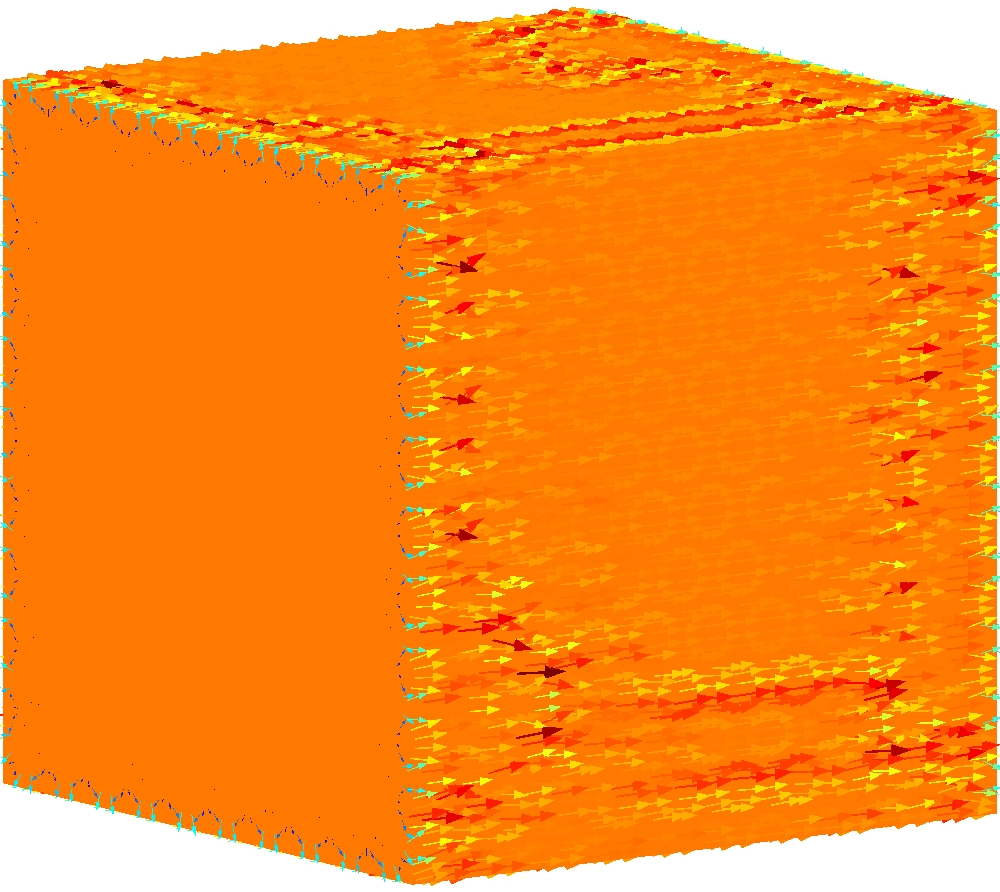}\\
  \caption{Approximations of the tangential component of $(1,0,0)$ in RWG (left) and BC (right) spaces on a discretised cube with 2202 edges.}
  \label{fig:f1_f2}
\end{subfigure}\\

\begin{subfigure}{\textwidth}
  \centering
  \includegraphics[scale=.15]{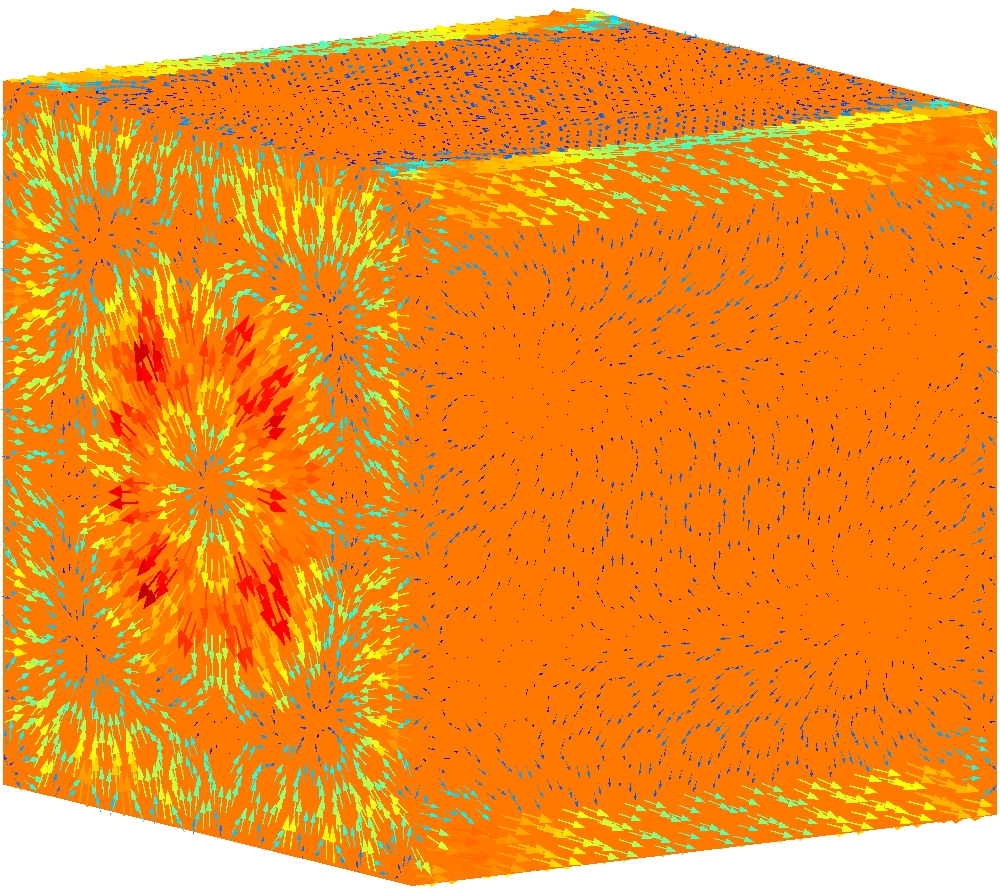}
  \includegraphics[scale=.15]{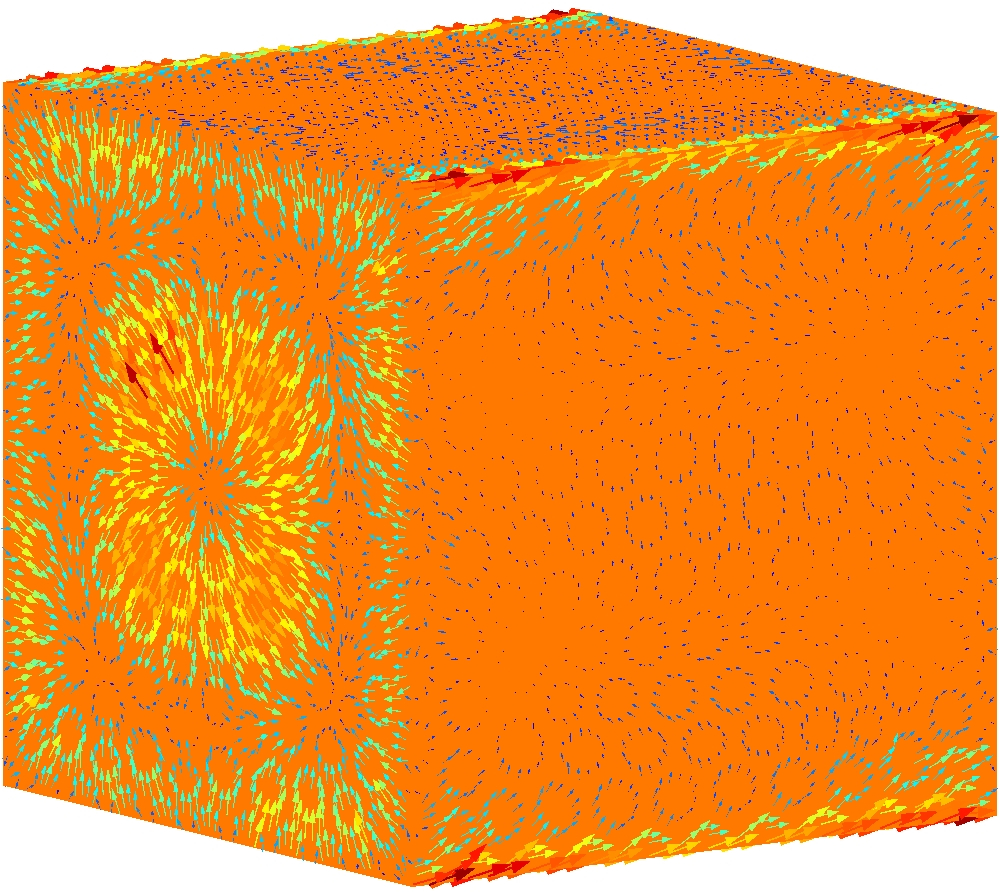}\\
  \caption{The result of applying the Calder\'on projector to the functions in Figure \ref{fig:f1_f2}. These functions
are (up to discretisation error) valid exterior Maxwell Cauchy data.}
  \label{fig:f1_f2_out}
\end{subfigure}
\caption{Visualisation of the Calder\'on projector applied to non-compatible Cauchy data.}
\label{fig:calderon_visualized}
\end{figure}

In Figure \ref{fig:calderon_visualized}, we show the electric and magnetic trace data obtained from taking the tangential trace of the vector $(1, 0, 0)$ (upper plot), and the result of applying the Calder\'on projector to the incompatible pair of Cauchy data (bottom plot).

\subsection{Implementational Details}

The discrete multitrace operator $A$ consists of the two magnetic field operator discretisations $H_1$ and $H_2$, and the two electric field operator discretisations $E_1$ and $E_2$. In practice, we only create two operators $\tilde{E}$, and $\tilde{H}$, using RWG basis functions
on the barycentrically refined grid. Let $\tilde{M}$ be the $L^2$ mass matrix associated with this RWG space, defined by
$$
\tilde{M}_{ij} = \int_{\Gamma} \mathbf{RWG}_i\cdot \mathbf{RWG}_j,
$$
where $\mathbf{RWG}_i$ is the $i$th RWG basis function on the barycentrically refined grid. Let now $M_{RWG}$ be the mass matrix obtained from trial functions in the RWG space on the barycentrically refined grid, and test functions from the original RWG space on the coarse grid. 
We correspondingly define the mass matrix $M_{BC}$
with test functions from the BC space on the original coarse grid. The operators $E_i$, and $H_i$, $i=1,2$, are now given as
\begin{align}
H_1 &= M_{BC}\tilde{M}^{-1}\tilde{H}\tilde{M}^{-1}M_{RWG}^T,\quad
E_1 = M_{BC}\tilde{M}^{-1}\tilde{H}\tilde{M}^{-1}M_{BC}^T\nonumber,\\
H_2 &= M_{RWG}\tilde{M}^{-1}\tilde{H}\tilde{M}^{-1}M_{BC}^T,\quad
E_2 = M_{RWG}\tilde{M}^{-1}\tilde{H}\tilde{M}^{-1}M_{RWG}^T.\nonumber
\end{align}
In \cite{BCPrecond}, a similar construction of the matrices is suggested. The difference is that there permutation matrices that represent the basis functions on the coarse mesh in terms of basis functions on the barycentric refinement are stated explicitly.

The implicit construction here has the advantage that it is independent of the particular space. All that is needed is the ability to construct mass matrices, which is often already available. A potential performance pitfall is the application of the mass matrix inverse of $\tilde{M}$ for each matrix vector product with $E_i$ or $E_j$. We automatically precompute the LU decomposition of $\tilde{M}$. Even for fairly large meshes this is done in a few seconds.

\section{Electric field integral equation (EFIE)}
\label{sec:efie}
We now turn our attention to the EFIE. The EFIE is widely used in applications for low-frequency scattering from closed and open surfaces. In its strong form the indirect EFIE to find the scattered field $\bE\scat$ in \eqref{eq:maxwell} can be written as (see \cite{BuffaHiptmair})
\begin{equation}
\bopE\B\lambda = \gt\exterior\bE\inc.
\label{eq:efie_indirect}
\end{equation}
Correspondingly, the direct EFIE can be written as
\begin{equation}
\bopE\B\pi = (\tfrac12\bopI+\bopH)\gt\exterior\bE\inc.
\label{eq:efie_direct}
\end{equation}

The direct EFIE is derived from the first line of the exterior Calder\'on projector. The unknown $\B\pi$ has a physical meaning as the Neumann trace of the scattered field. Once, the solution $\B\lambda$ for the indirect EFIE or $\B\pi$ for the direct EFIE is available the solution $\bE\scat$ can be computed using
\begin{equation}
\label{eq:efie_indirect_rep}
\bE\scat = -\mathcal{E}\B\lambda
\end{equation}
for the indirect EFIE, and
\begin{equation}
\label{eq:efie_direct_rep}
\bE\scat = \mathcal{H}\gt\exterior\bE\inc-\mathcal{E}\B\pi
\end{equation}
for the direct EFIE. 

By itself the EFIE is ill-conditioned, making necessary either direct solvers or efficient preconditioners. The Calder\'on identities 
described in Section \ref{sec:calderon} provide an efficient preconditioning strategy. From the top-left block of $\bopA^2=\frac14\bopI$ it 
follows that $\bopE^2=-\tfrac14\bopI+\bopH^2$. The eigenvalues of $\bopE$ accumulate at $0$ and $\infty$ making discretisations of this 
operator highly ill-conditioned. However, the operator $\bopH$ is compact on smooth surfaces \cite[Section 5.5]{NedelecBook}. Hence, the 
eigenvalues of $\bopE^2$ accumulate at $-\frac{1}{4}$. Implementations of this self-regularising property depend on the ability to perform 
the operator product $\bopE^2$ in a stable way. In Section \ref{sec:calderon}, we saw how stable operator products are implemented, 
based on Buffa--Christiansen bases, and this property will be used here to implement the preconditioning strategy. Further details on the 
Calder\'on preconditioner for the EFIE can also be found in \cite{BCPrecond}.

The numerical implementation of the EFIE is typically based on discretising the operator $\bopE$ in \eqref{eq:efie_indirect} and \eqref{eq:efie_direct} by pairs of 
RWG/SNC trial and dual spaces. For the direct EFIE, this implies that a discretisation of the Calder\'on projector is used, in which the 
tangential trace data is represented with BC basis functions and the Neumann trace data is represented with RWG basis functions.

\begin{figure}
\centering
\begin{lstlisting}[language=Python]
import bempp.api
import numpy as np

from bempp.api.operators.boundary.maxwell \
    import multitrace_operator
from bempp.api.operators.potential.maxwell \
    import electric_field

grid = ...
k = ...

def incident_field(x):
    return np.array([np.exp(1j * k * x[2]), 
                     0. * x[2], 0. * x[2]])

def tangential_trace(x, n, domain_index, result):
    result[:] = np.cross(incident_field(x), n, axis=0)    
multitrace = multitrace_operator(grid, k)
bc_space = multitrace.range_spaces[1]
snc_space = multitrace.dual_to_range_spaces[1]
    
grid_fun = bempp.api.GridFunction(
    bc_space, fun=tangential_trace,
    dual_space=snc_space)

E2 = -multitrace[1, 0]
E1 = multitrace[0, 1]
op = E1 * E2
rhs = E1 * grid_fun

sol, info = bempp.api.linalg.gmres(
    op, rhs, use_strong_form=True)
eval_points = ...
efie_pot = electric_field(sol.space, eval_points, k)
field = -efie_pot * sol
\end{lstlisting}
\caption{Code snippet to solve a preconditioned
electric field problem.}
\label{fig:efie_preconditioned}
\end{figure}

In Figure \ref{fig:efie_preconditioned} we show the BEM++ implementation of the Calder\'on preconditioned indirect EFIE, based on the stable formulation of the multitrace operator $\bopA$. The EFIE operator is $E_2$ from the discretisation of $\bopA$, which maps from the RWG to the BC space. The preconditioning operator is $E_1$, which maps from the BC space to the RWG space. Correspondingly, the right-hand side incident wave is discretised using BC functions to be compatible with the preconditioning operator. The solution \verb@sol@ lives in the RWG space, which is the domain space of $E_2$. The GMRES routine has the
additional parameter \verb@use_strong_form=True@.
This enables mass matrix preconditioning and further reduces the number of iterations.
At the end we discretise a domain potential operator
over the RWG space of the solution function \verb@sol@.
This evaluates the potential at arbitrary points in the
domain.

We emphasise that BEM++ automatically only
assembles a single electric field operator on the barycentric refinement of the grid. The magnetic components of the multitrace operator are not assembled as they are not needed for the solution of the problem.

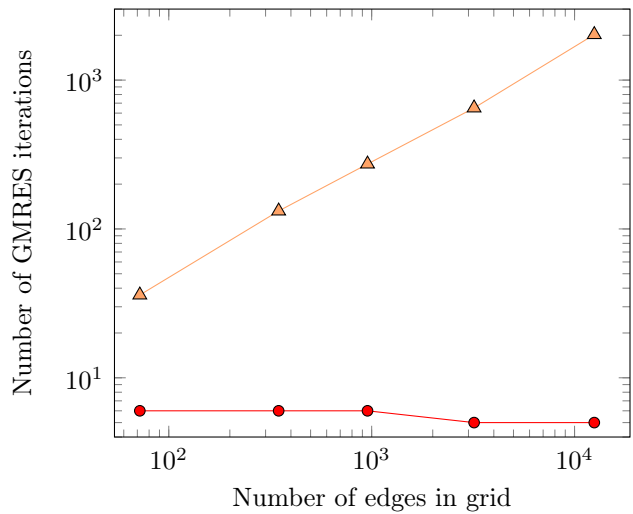
\begin{figure}
  \centering
\begin{tikzpicture}

\begin{axis}[
xlabel={Number of edges in grid},
ylabel={Number of GMRES iterations},
xmin=54, xmax=18765,
ymin=4, ymax=3000,
ymode=log,
xmode=log,
axis on top
]
\efieplot
table {%
72 36
348 132
954 273
3198 650
12510 2022
};
\calefieplot
table {%
72 6
348 6
954 6
3198 5
12510 5
};
\end{axis}

\end{tikzpicture}
  \caption{The number of GMRES iterations taken to solve the EFIE (\efiedesc) and preconditioned EFIE (\calefiedesc) to a tolerance of $1\times10^{-5}$ for scattering from the unit sphere.}
  \label{fig:efie_calderon_iteration_count}
\end{figure}

Taking $\Omega$ as unit sphere,
the incident wave $\displaystyle\bE\inc=\left[e^{\ii kz},0,0\right]$, and $k=2$, we discretise the EFIE on a
series of triangular grids with different levels of refinement.
The number of GMRES iterations required to solve the linear system arising from EFIE 
and the Calder\'on preconditioned EFIE applied to this example problem are shown in Figure \ref{fig:efie_calderon_iteration_count}.
It can be seen that while the number of iterations required to solve the EFIE rises quickly as the grid is refined, the number of
iterations required to solve the preconditioned EFIE remains below 10.

\begin{figure}
  \centering
  \includegraphics[width=.49\textwidth]{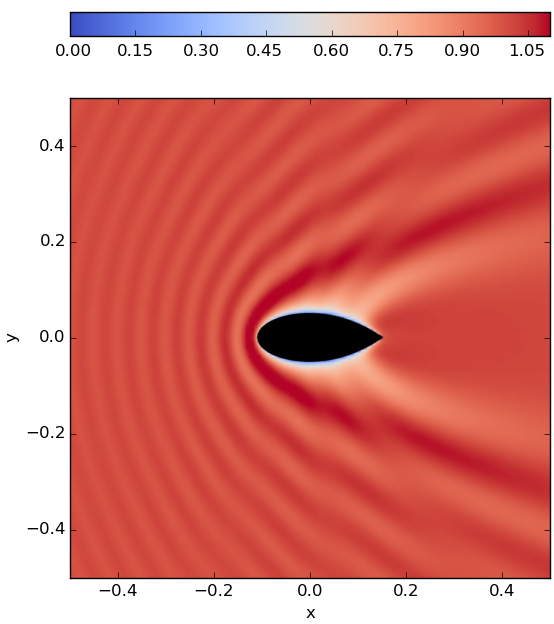}
\begin{tikzpicture}

\begin{axis}[small,
xlabel={Number of iterations},
ylabel={GMRES residual},
xmin=0, xmax=100,
ymin=1e-14, ymax=0.5,
ymode=log,
axis on top
]
\calefieplot[1]
table {%
1 0.40407534640473375
2 0.13315995682543402
3 0.06016670930011505
4 0.04091557042567667
5 0.030239701060197445
6 0.024651071571276367
7 0.02070480233386302
8 0.015327843867037361
9 0.013294468791739622
10 0.01073856738610812
11 0.005835163727455261
12 0.002753037098407198
13 0.0013445526407057842
14 0.000644014921498786
15 0.00027230214301466216
16 0.00014954995320887594
17 8.2461160303912e-05
18 3.4825452636017744e-05
19 2.186613031057502e-05
20 1.2873673242250206e-05
21 8.679576287184778e-06
22 6.732006625036935e-06
23 5.03144168761512e-06
24 3.743051469737036e-06
25 2.9261901620759574e-06
26 2.1728861301831174e-06
27 1.8040017190946769e-06
28 1.5481668356512497e-06
29 1.3588339069555332e-06
30 1.1557776766662453e-06
31 8.983953626629441e-07
32 6.697790076663105e-07
33 4.052785395976176e-07
34 2.3057931284019592e-07
35 1.479668872553349e-07
36 9.845483622200893e-08
37 7.640860723635373e-08
38 5.986931611040856e-08
39 4.5548853517227854e-08
40 3.539106912788732e-08
41 3.2094648806576474e-08
42 2.885857709604959e-08
43 2.606967337693494e-08
44 2.32649051115391e-08
45 2.0599072460164444e-08
46 1.771709442577409e-08
47 1.5070784868694207e-08
48 1.2812474894846971e-08
49 1.1175234454220843e-08
50 9.811120242749944e-09
51 8.153107479419397e-09
52 6.8977430142019424e-09
53 5.3840340194052276e-09
54 3.5939313747211564e-09
55 2.189574669770281e-09
56 1.3132772177627565e-09
57 8.364357392156851e-10
58 6.040171009939215e-10
59 4.77858973517966e-10
60 3.9481877956270154e-10
61 3.499987381345721e-10
62 3.131984341997588e-10
63 2.5006421878757754e-10
64 1.839912735605551e-10
65 1.1409355614787115e-10
66 8.722672053332117e-11
67 6.934115832823475e-11
68 5.759728730580402e-11
69 4.578652640516436e-11
70 3.3981402403483545e-11
71 2.5412822894555624e-11
72 1.8661246598216527e-11
73 1.4133417484017396e-11
74 1.1988252846865798e-11
75 9.536309152945562e-12
76 8.425650206472077e-12
77 7.27740080506799e-12
78 6.361664886390831e-12
79 5.481797427170449e-12
80 4.328680952228398e-12
81 4.02282320681923e-12
82 3.4207723812093338e-12
83 2.9719125579579313e-12
84 2.6231356346122204e-12
85 2.28538733338193e-12
86 2.0246562505309625e-12
87 1.7212661608906702e-12
88 1.429000093259403e-12
89 1.1591171409722968e-12
90 9.263976143657398e-13
91 7.399026583213624e-13
92 5.81731089334694e-13
93 3.994569447752317e-13
94 2.7659695463704624e-13
95 1.9868916328155379e-13
96 1.3963405179350144e-13
97 1.0858906375487353e-13
98 8.910276430646953e-14
99 7.371423133310741e-14
100 6.97089762701337e-14
};
\end{axis}

\end{tikzpicture}
  \caption{Slice at $z=0$ of squared electric field density of the wave $\bE\inc=[0,0,e^{\ii kx}]$, with $k=20\pi$, scattering off the NASA
almond, computed using the indirect preconditioned EFIE discretised on a grid with 2442 edges; and the corresponding convergence of the GMRES residual.}
  \label{fig:efie_almond_plots}
\end{figure}

\begin{figure}
  \centering
  \includegraphics[width=.49\textwidth]{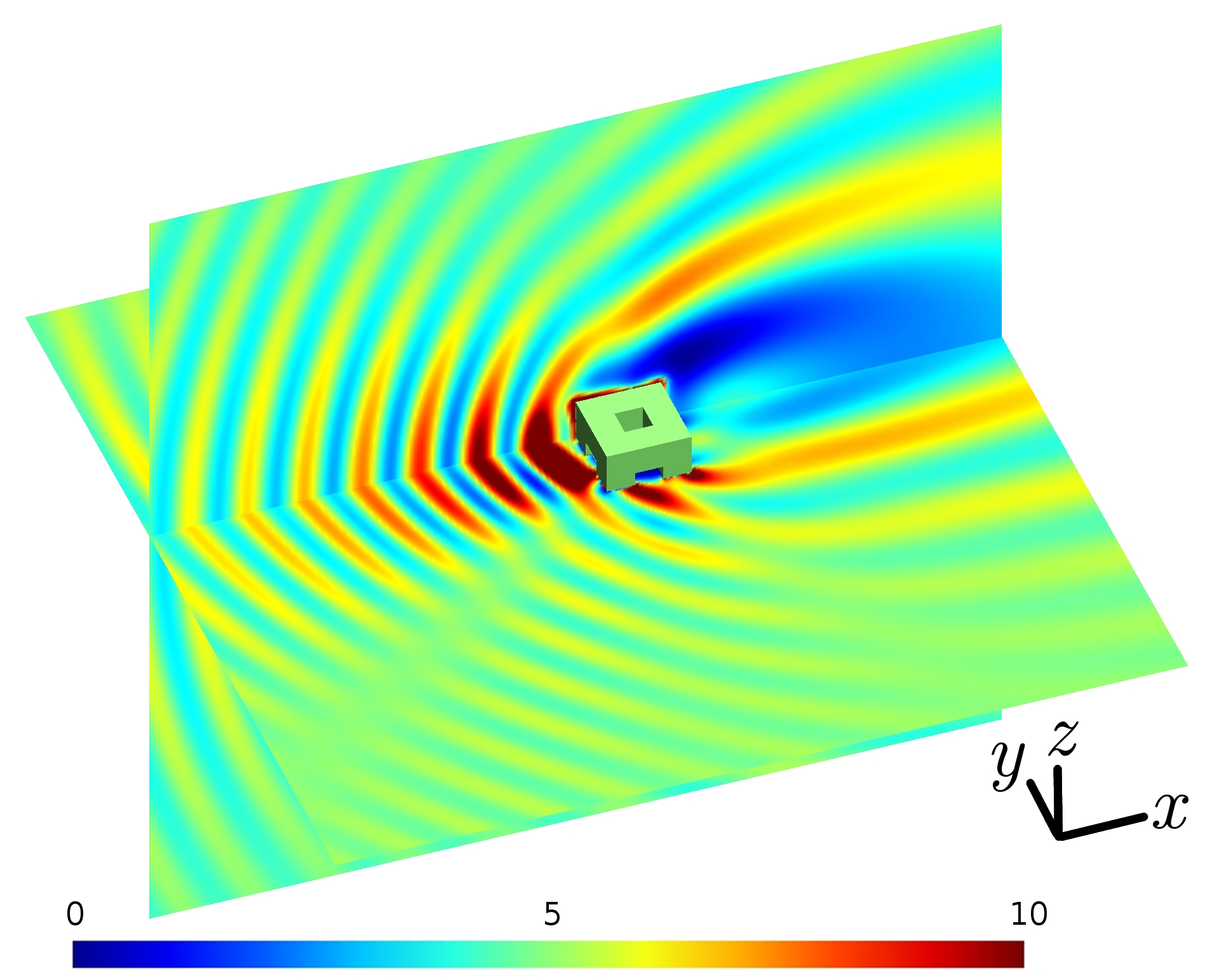}
\begin{tikzpicture}

\begin{axis}[small,
xlabel={Number of iterations},
ylabel={GMRES residual},
xmin=0, xmax=60,
ymin=1e-14, ymax=0.5,
ymode=log,
axis on top
]
\calefieplot[1]
table {%
1 0.6005434517950654
2 0.5008792241817531
3 0.25627547289896035
4 0.10964585977258218
5 0.06018602994899594
6 0.030009625629529735
7 0.026738277522934974
8 0.021441130348257537
9 0.010063119155161483
10 0.003641212629393378
11 0.001229209866210397
12 0.00042162916871510273
13 0.00035347781355950285
14 0.00025911168729377215
15 0.00011762512563568108
16 4.095887691452093e-05
17 1.5380054002761303e-05
18 1.3519109638397611e-05
19 9.63350196004712e-06
20 6.22000184251305e-06
21 4.57945464200252e-06
22 3.4304342126395165e-06
23 2.7698703458161614e-06
24 1.6265918167738813e-06
25 7.18944767907409e-07
26 3.2385112417929944e-07
27 2.9645744682868e-07
28 1.6535852814598367e-07
29 6.966768548491224e-08
30 2.0662026689467678e-08
31 7.077652141606951e-09
32 3.3721827461260604e-09
33 2.7735785924276746e-09
34 1.6683176361958699e-09
35 9.034691921933553e-10
36 3.8424123459748676e-10
37 1.8738818328676243e-10
38 1.455929939620588e-10
39 1.130138608700053e-10
40 5.975183555119924e-11
41 4.600802438117222e-11
42 3.512355948715719e-11
43 2.1617161042237566e-11
44 1.244124354171174e-11
45 6.700086470593679e-12
46 2.943105692607921e-12
47 2.6591251963358866e-12
48 2.0273079541896e-12
49 1.0340189979254972e-12
50 3.9199392417198925e-13
51 1.1109060831363764e-13
52 4.446406386739378e-14
53 4.11825876124667e-14
54 2.349089491069821e-14
55 1.345512151903681e-14
56 4.262061593101495e-15
57 2.3651426203819956e-15
58 1.9873497224109714e-15
59 1.1828885367734603e-15
60 5.496481774328152e-16
};
\end{axis}

\end{tikzpicture}
  \caption{Slices at $z=0.5$ and $y=0.5$ of squared electric field density of the wave $\bE\inc=\B{p}e^{\ii k\B{d}\cdot\B{x}}$, with $\B{p}=[-1,2,0]$,
$\B{d}=[\tfrac2{\sqrt5},\tfrac1{\sqrt5},0]$, and $k=5$ scattering off a level 1 Menger sponge; and the corresponding convergence of the GMRES residual.
This was computed using the indirect preconditioned EFIE discretised on a grid with 4680 edges.}
  \label{fig:efie_sponge_plots}
\end{figure}

Figures \ref{fig:efie_almond_plots} and \ref{fig:efie_sponge_plots} show electromagnetic waves scattering off more interesting obstacles.
Both were computed using the indirect Calder\'on preconditioned EFIE.

The left plot of Figure \ref{fig:efie_almond_plots} shows the wave $\bE\inc=[0,0,e^{\ii kx}]$, with $k=20\pi$, scattering off 
the NASA almond benchmarking shape, as defined in \cite{almond}.
This was discretised on a grid with 2442 edges. The right plot shows the corresponding convergence curve of the GMRES residual for the preconditioned EFIE.

Figure \ref{fig:efie_sponge_plots} shows the wave $\bE\inc=\B{p}e^{\ii k\B{d}\cdot\B{x}}$, with $\B{p}=[-1,2,0]$,
$\B{d}=[\tfrac2{\sqrt5},\tfrac1{\sqrt5},0]$, and $k=5$, scattering off a level 1 Menger sponge, and the corresponding convergence curve.
The Menger sponge was discretised using a grid with 4680 edges.

\section{Magnetic field integral equation (MFIE)}
\label{sec:mfie}

The MFIE can be represented in BEM++ as easily as the EFIE. The strong form of the indirect MFIE is (see \cite{BuffaHiptmair})
$$
(\bopH-\tfrac12\bopI)\B\xi = \gt\exterior\bE\inc.
$$
It is a valid formulation on closed domains. Its advantage compared to the EFIE is that on smooth domains, it is a compact perturbation of 
the identity operator and therefore well suited to iterative solvers. However, the robust implementation of the MFIE also on non-smooth 
domains requires the use of RWG trial spaces and RBC test spaces (see \cite{AndriulliMFIE}).
\begin{figure}
\centering
\begin{lstlisting}[language=Python]
from bempp.api.operators.potential.maxwell import \
	magnetic_field

calderon = ...
tangential_trace = ...
rwg_space = calderon.domain_spaces[0]
rbc_space = calderon.dual_to_range_spaces[0]

rhs = bempp.api.GridFunction(
    rwg_space, fun=tangential_trace,
    dual_space=rbc_space)
op = -calderon[0, 0]
sol, info = bempp.api.linalg.gmres(
    op, rhs, use_strong_form=True)

eval_points = ...
mfie_pot = magnetic_field(sol.space, eval_points, k)
field = -mfie_pot * sol    
\end{lstlisting}
\caption{Code snippet for the implementation of the MFIE in BEM++.}
\label{fig:mfie_code}
\end{figure}
The
MFIE can be easily implemented as shown in the code snippet in Figure \ref{fig:mfie_code}.
\begin{figure}
  \centering
\begin{tikzpicture}

\begin{axis}[small,
xlabel={Number of iterations},
ylabel={GMRES residual},
xmin=0, xmax=100,
ymin=1e-12, ymax=1,
ymode=log,
axis on top
]
\mfieplot[1]
table {%
1 0.09739653196936025
2 0.03246810831347021
3 0.009786738033018447
4 0.0030839625152994004
5 0.0011366438013882079
6 0.0003815631940876693
7 0.00015244836663398255
8 8.353385883598998e-05
9 6.524189361826505e-05
10 4.5902360238762704e-05
11 3.184764532861063e-05
12 2.292730739940691e-05
13 1.6910506039125556e-05
14 1.5133049931808173e-05
15 1.3292261975638115e-05
16 1.1562460695959248e-05
17 1.0902788943141362e-05
18 1.0257234608799893e-05
19 9.550625864945403e-06
20 9.257195090041696e-06
21 9.040168938692394e-06
22 8.509567106226216e-06
23 7.909221802197972e-06
24 7.346161103865587e-06
25 6.725666103373564e-06
26 6.193997238850087e-06
27 5.5807154138183606e-06
28 5.1333930677416775e-06
29 4.615276478812022e-06
30 4.197080996831331e-06
31 3.7613263916912378e-06
32 3.3356534050698814e-06
33 2.8903885027561653e-06
34 2.5634680674906e-06
35 2.3395055591208718e-06
36 2.0831443654657926e-06
37 1.835722019703345e-06
38 1.5977118491692398e-06
39 1.3458334920418708e-06
40 1.148483974922951e-06
41 9.710478568852977e-07
42 8.39397605066879e-07
43 7.636117416258557e-07
44 6.86007102535829e-07
45 6.224855298123087e-07
46 5.736935450304823e-07
47 5.413675714971276e-07
48 4.953993829127822e-07
49 4.557295055933419e-07
50 4.058256648888562e-07
51 3.6213085901987957e-07
52 3.313115352583064e-07
53 2.8965602064038484e-07
54 2.5275046035259976e-07
55 2.2444271487627977e-07
56 2.0026557606742732e-07
57 1.8344329971256637e-07
58 1.680299372905009e-07
59 1.598094132987287e-07
60 1.5420397151435922e-07
61 1.5071107001442065e-07
62 1.4690285983792779e-07
63 1.4090218770397926e-07
64 1.355307850483773e-07
65 1.2764782793591835e-07
66 1.188008741713906e-07
67 1.0724520999748321e-07
68 9.450666468311103e-08
69 8.494296138074893e-08
70 7.588673567433487e-08
71 7.078956874863148e-08
72 6.696491881331344e-08
73 6.165940604753656e-08
74 5.8157101204466446e-08
75 5.4566050438164615e-08
76 5.062981855821108e-08
77 4.591028122699917e-08
78 4.141008509415699e-08
79 3.7355097378833335e-08
80 3.396533727360597e-08
81 3.1184710014922826e-08
82 2.9118816431371965e-08
83 2.7150552601676813e-08
84 2.5503195882374288e-08
85 2.1713068298981616e-08
86 1.8492266632716734e-08
87 1.6727152773797976e-08
88 1.5162387168958242e-08
89 1.4036072858976136e-08
90 1.3273367271203362e-08
91 1.253300230630197e-08
92 1.1628530813635158e-08
93 1.0937010888598676e-08
94 1.0407165774377192e-08
95 9.716563946500327e-09
96 9.203478485143208e-09
97 8.73202560191937e-09
98 8.319948977269482e-09
99 7.709554690253309e-09
100 6.961283052684614e-09
101 6.220732999964805e-09
102 5.674875245381469e-09
103 5.289932836883209e-09
104 4.8955813234703e-09
105 4.606353765166653e-09
106 4.234201860477365e-09
107 3.907999280727865e-09
108 3.528914461560126e-09
109 3.2197137595326165e-09
110 2.976322889020191e-09
111 2.670622516970971e-09
112 2.4380883267309855e-09
113 2.1552205861445946e-09
114 1.8627942269114213e-09
115 1.7042277046445384e-09
116 1.5312300921269263e-09
117 1.3102483435949176e-09
118 1.0541979233477689e-09
119 8.643714334110835e-10
120 7.002021751083676e-10
121 6.137236560026666e-10
122 5.3906102871202e-10
123 4.594487578650781e-10
124 4.138119860500616e-10
125 3.639475475706182e-10
126 3.194464234020067e-10
127 2.934506537175618e-10
128 2.685038397261607e-10
129 2.465467139629879e-10
130 2.2987458140939485e-10
131 2.0959305607719208e-10
132 1.9142721187706744e-10
133 1.6915394965123868e-10
134 1.5345877800850188e-10
135 1.365990658439059e-10
136 1.2088648990010133e-10
137 1.0402325190658222e-10
138 9.022502232149417e-11
139 7.737637785742396e-11
140 6.334798804218193e-11
141 5.4061520808261774e-11
142 4.7757044895806946e-11
143 4.241018588593862e-11
144 3.8787600819244726e-11
145 3.501116490101617e-11
146 3.2019067864757735e-11
147 2.9846981141871654e-11
148 2.7142294333565754e-11
149 2.5194245117279852e-11
150 2.2747603309753056e-11
151 2.0561308957762252e-11
152 1.8932696103852327e-11
153 1.6680836218496497e-11
154 1.508442108048266e-11
155 1.4336240942713433e-11
156 1.3389077894344404e-11
157 1.2597682860461825e-11
158 1.190382826269065e-11
159 1.0959384467353563e-11
160 9.998285103198434e-12
161 9.255739955926586e-12
162 8.758666122923383e-12
163 8.404865701662256e-12
164 8.005710837330747e-12
165 7.43217289556426e-12
166 7.094852229624942e-12
167 6.4734502290367635e-12
168 5.99058344541017e-12
169 5.517747653086152e-12
170 4.990515181124242e-12
171 4.585769620983278e-12
172 4.198746819249849e-12
173 3.835550175739793e-12
174 3.57542092947208e-12
175 3.314051748475639e-12
176 3.0994475157764493e-12
177 2.8236883817488644e-12
178 2.604386250757484e-12
179 2.323165166259958e-12
180 2.013719465067691e-12
181 1.8007829447209784e-12
182 1.607433314279581e-12
183 1.407459211605283e-12
184 1.284175207956195e-12
185 1.1658005291607934e-12
186 1.0555057518566998e-12
187 9.809214256657037e-13
188 8.941791484407344e-13
189 8.253625027524133e-13
190 7.688109884090803e-13
191 7.140219215918101e-13
192 6.661536851234008e-13
193 5.989067927510902e-13
194 5.460056491056177e-13
195 4.944257143803116e-13
196 4.347272323989612e-13
197 3.8184804938289935e-13
198 3.3473735288399336e-13
199 2.9129340611503285e-13
200 2.4715773778065973e-13
201 2.1541927301522658e-13
202 1.8845413830763308e-13
203 1.6757264807356638e-13
204 1.4859796688661457e-13
205 1.308713825381872e-13
206 1.187394761238723e-13
207 1.0830547353225468e-13
208 9.817755043085349e-14
209 9.153296853382121e-14
210 8.521872133937572e-14
211 8.010953418814656e-14
212 7.732557449122808e-14
213 7.266539672955066e-14
214 6.933898122724827e-14
215 6.385484563330094e-14
216 5.824359127005478e-14
217 5.396238324989138e-14
218 4.950214492753883e-14
219 4.637397833506537e-14
220 4.2723647338800063e-14
221 3.98042549340172e-14
222 3.807327765211632e-14
223 3.590000208647268e-14
224 3.39300563663881e-14
225 3.146997307006937e-14
226 2.90462397798808e-14
227 2.7402720067307343e-14
228 2.5134409734862134e-14
229 2.356563403832696e-14
230 2.1896647951523378e-14
231 2.0024073039295508e-14
232 1.869543472075138e-14
233 1.7351016480945718e-14
234 1.588713216480083e-14
235 1.5033717104577687e-14
236 1.37006037525059e-14
237 1.2994687938832992e-14
238 1.2095028895903607e-14
239 1.1227517895711519e-14
240 9.81580959004646e-15
241 8.562245185143763e-15
242 7.901528813234886e-15
243 7.238039383387934e-15
244 6.613919965708521e-15
245 6.236931174805335e-15
246 5.849421472265195e-15
247 5.478177588172701e-15
248 5.191983637286101e-15
249 4.859732560344419e-15
250 4.385504398130441e-15
251 4.032754738349397e-15
252 3.628173067212265e-15
253 3.197957433511041e-15
254 2.9107812598409874e-15
255 2.5828535912444035e-15
256 2.269661927307058e-15
257 1.9707423217501338e-15
258 1.6423442342016485e-15
259 1.2930996214266393e-15
260 1.4034923464165288e-15
261 9.850398701588959e-16
};
\end{axis}

\end{tikzpicture}
\begin{tikzpicture}

\begin{axis}[small,
xlabel={Number of iterations},
xmin=0, xmax=100,
ymin=1e-12, ymax=1,
ymode=log,
axis on top
]
\mfieplot[1]
table {%
1 0.8725472331969267
2 0.6429582756379736
3 0.5115398326529771
4 0.35836849568920215
5 0.25788307603690286
6 0.1708115461130254
7 0.103688820148429
8 0.06409293783201524
9 0.0397896858394414
10 0.030518946204696557
11 0.029547801627909574
12 0.02233555279702101
13 0.013922135214589022
14 0.01035221720970957
15 0.008919249916670247
16 0.00680948334607921
17 0.004449493103980392
18 0.0028836174054650185
19 0.0017899755888339222
20 0.0010190513392753517
21 0.0007619371530433409
22 0.0006995126870973485
23 0.0005823882852829413
24 0.0005191202556539587
25 0.0004831545725333321
26 0.00037334994012913267
27 0.0002773360346873486
28 0.0001999367089490265
29 0.00014515029294091526
30 0.0001009078580268158
31 6.406643345117206e-05
32 3.9307105779649215e-05
33 2.3957689656615274e-05
34 1.5626900651018103e-05
35 1.3474672519555846e-05
36 1.2651290008794378e-05
37 1.0339275682943102e-05
38 8.13502569522773e-06
39 6.894507829185476e-06
40 5.02773097504499e-06
41 3.865775689008626e-06
42 3.2544799447070472e-06
43 2.6871717543391826e-06
44 2.159517129780863e-06
45 1.9789281963925502e-06
46 1.8283570170016695e-06
47 1.2962985290779144e-06
48 8.603436722007412e-07
49 5.496626371863151e-07
50 3.3274265018257703e-07
51 2.1718359671883225e-07
52 1.4610816759247498e-07
53 1.0669346460138552e-07
54 8.116496755354499e-08
55 6.135344205017637e-08
56 5.4698858988487383e-08
57 4.9109663977497564e-08
58 4.276298766710674e-08
59 3.7421817120034164e-08
60 2.5040898486382613e-08
61 1.917076722038099e-08
62 1.715711708500136e-08
63 1.5003084562059637e-08
64 1.3542800110535057e-08
65 1.2541175615515089e-08
66 9.768279994125722e-09
67 7.039637678725661e-09
68 5.36393325862913e-09
69 3.81552452703874e-09
70 2.606297018442209e-09
71 1.7632572631878174e-09
72 1.1624984184173707e-09
73 7.412766222450079e-10
74 4.657166488406153e-10
75 3.980896535603539e-10
76 3.6783155411185965e-10
77 3.006448371009232e-10
78 2.3139578062832752e-10
79 1.942422606409575e-10
80 1.4435808719153855e-10
81 1.0830758247405199e-10
82 8.94659894733267e-11
83 7.31588191870836e-11
84 6.121919634365556e-11
85 5.6345040083681065e-11
86 5.012899063508748e-11
87 3.319141989480455e-11
88 2.067425130976743e-11
89 1.3426000315974425e-11
90 8.559368411500505e-12
91 5.7593512817299455e-12
92 3.957255173282492e-12
93 2.904516757073118e-12
94 2.131154846128395e-12
95 1.5583449058289828e-12
96 1.4109491340926117e-12
97 1.2995322764561258e-12
98 1.1587435081979005e-12
99 1.0251377985511743e-12
100 7.337237709854734e-13
101 5.786369138479023e-13
102 5.196321377784388e-13
103 4.559908552562459e-13
104 4.1173113636179726e-13
105 3.7726006298098724e-13
106 2.9173115325411254e-13
107 2.02683635817119e-13
108 1.5435003054487572e-13
109 1.1131059577014781e-13
110 7.639206861402075e-14
111 5.1999786670612315e-14
112 3.36125987314732e-14
113 2.163568035526185e-14
114 1.3454735311778124e-14
115 1.1918099419316101e-14
116 1.0437423289881187e-14
117 8.61622685489726e-15
118 7.19693020330758e-15
119 5.875937774863151e-15
120 4.2351111614128626e-15
121 2.999774344409779e-15
122 2.4639538788112906e-15
123 2.2251125945438596e-15
124 2.006457981428879e-15
125 1.824517574777971e-15
126 1.4990293604470443e-15
127 9.5009042370323e-16
};
\end{axis}

\end{tikzpicture}
  \caption{GMRES residuals vs iterations for the MFIE for the NASA almond (left) and Menger sponge (right) examples.}
  \label{fig:mfie_convergences}
\end{figure}
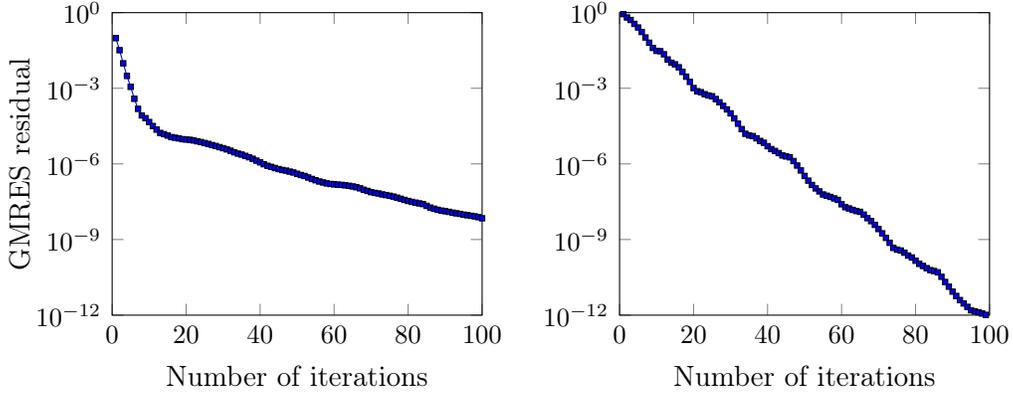

In Figure \ref{fig:mfie_convergences}, we demonstrate the rate of convergence of the MFIE for solving the NASA almond and the Menger sponge examples 
from Section \ref{sec:efie}. Both converge nicely without additional preconditioning. While the rate of convergence is slower than in the 
preconditioned EFIE case, only one boundary operator needs to be applied for each iteration step, while the preconditioned EFIE requires 
two applications.

\section{Combined field integral equation (CFIE)}
\label{sec:cfie}
While the EFIE and MFIE are efficient for low-frequency Maxwell problems, they lead to break-down close to interior resonances. The CFIE is 
immune to breakdown at resonances and is therefore particularly suitable for high-frequency scattering problems. Here, we focus on the 
direct CFIE and the stable version of it derived in \cite{rokhlin_cfie}. In its strong form, it is given as
\begin{equation}
\label{eq:cfie}
\left(-\bopR\bopE+\tfrac12\bopI+\bopH\right)\B\pi
    = -\bopR(\tfrac12\bopI+\bopH)\gt\exterior\bE\inc-\bopE\gt\exterior\bE\inc.
\end{equation}
The CFIE is a linear combination of the direct MFIE, obtained from the second row of the exterior Calder\'on projector, and a regularised 
direct EFIE, obtained by multiplying the first row of the exterior Calder\'on projector with a regularisation operator $\bopR$. Frequently, the 
EFIE component is multiplied with a complex scalar. This is not necessary here, as in our implementation the electric field operator itself 
is already scaled with $\ii$.

\begin{figure}
\centering
\begin{lstlisting}[language=Python]
multitrace = ...
multitrace_scaled = ...
identity = ...
tangential_trace = ...

rwg_space = multitrace.domain_spaces[0]
snc_space = multitrace.dual_to_range_spaces[1]
bc_space = multitrace.domain_spaces[1]
rbc_space = multitrace.dual_to_range_spaces[0]

calderon = 0.5 * identity + multitrace
grid_fun = bempp.api.GridFunction(
    bc_space, fun=tangential_trace,
    dual_space=snc_space)

R = multitrace_scaled[0, 1]
E1 = multitrace[0, 1]
E2 = -multitrace[1, 0]
mfie1 = calderon[0, 0]
mfie2 = calderon[1, 1]

rhs = -R * mfie2 * grid_fun - E1 * grid_fun
op = -R * E2 + mfie1
sol, info = bempp.api.linalg.gmres(
    op, rhs, use_strong_form=True)
\end{lstlisting}
\caption{Code snippet for the implementation of the CFIE in BEM++.}
\label{fig:cfie_code}
\end{figure}

We use the RWG space for the unknown Neumann trace $\B\pi$. Hence, we swap $E_1$ and $E_2$, and $H_1$ and $H_2$ in the discretisation of the Calder\'on projector in Section \ref{sec:calderon}.

It follows that we discretise $\bopH$ on the left-hand side of \eqref{eq:cfie} with $H_1$.
Moreover, for $\bopE$ on the left-hand side we choose the discrete operator $E_2$. The operator $E_2$ maps from RWG into BC, while $H_1$ maps from RWG into RWG.
We therefore require that a discretisation of $\bopR$ maps from the BC space to the RWG space. We could for example choose the operator $E_1$. But this operator is not injective at interior electric eigenvalues. The solution is to choose $E_1$ based on the wavenumber $\ii k$, instead of $k$ (see \cite{rokhlin_cfie}). On the right-hand side we choose for $\bopH$ the discretisation $H_2$ and for $\bopE$ the discretisation $E_1$ to stay compatible with the corresponding direct EFIE and direct MFIE formulation.
We can easily implement this in the framework of BEM++ with the code snippet in Figure \ref{fig:cfie_code}.

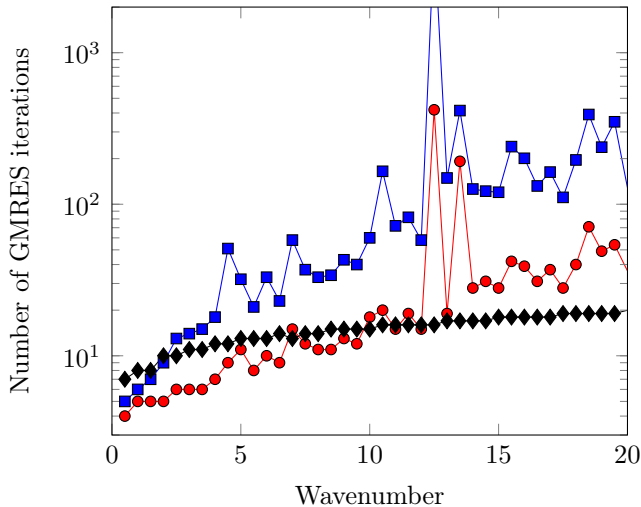
\begin{figure}
  \centering
\begin{tikzpicture}

\begin{axis}[
xlabel={Wavenumber},
ylabel={Number of GMRES iterations},
xmin=0, xmax=20,
ymin=3, ymax=2000,
ymode=log,
axis on top
]
\calefieplot
table {%
0.5 4
1.0 5
1.5 5
2.0 5
2.5 6
3.0 6
3.5 6
4.0 7
4.5 9
5.0 11
5.5 8
6.0 10
6.5 9
7.0 15
7.5 12
8.0 11
8.5 11
9.0 13
9.5 12
10.0 18
10.5 20
11.0 15
11.5 19
12.0 15
12.5 420
13.0 19
13.5 192
14.0 28
14.5 31
15.0 28
15.5 42
16.0 39
16.5 31
17.0 37
17.5 28
18.0 40
18.5 71
19.0 49
19.5 54
20.01 36
};
\mfieplot
table {%
0.5 5
1.0 6
1.5 7
2.0 9
2.5 13
3.0 14
3.5 15
4.0 18
4.5 51
5.0 32
5.5 21
6.0 33
6.5 23
7.0 58
7.5 37
8.0 33
8.5 34
9.0 43
9.5 40
10.0 60
10.5 165
11.0 72
11.5 82
12.0 58
12.5 6000
13.0 149
13.5 415
14.0 126
14.5 122
15.0 120
15.5 240
16.0 201
16.5 132
17.0 163
17.5 111
18.0 196
18.5 391
19.0 238
19.5 350
20.01 126
};
\cfieplot
table {%
0.5 7
1.0 8
1.5 8
2.0 10
2.5 10
3.0 11
3.5 11
4.0 12
4.5 12
5.0 13
5.5 13
6.0 13
6.5 14
7.0 13
7.5 14
8.0 14
8.5 15
9.0 15
9.5 15
10.0 15
10.5 16
11.0 16
11.5 16
12.0 16
12.5 16
13.0 17
13.5 17
14.0 17
14.5 17
15.0 18
15.5 18
16.0 18
16.5 18
17.0 18
17.5 19
18.0 19
18.5 19
19.0 19
19.5 19
20.01 20
};
\end{axis}

\end{tikzpicture}
  \caption{The number of GMRES iterations needed to solve the preconditioned EFIE (\calefiedesc), MFIE (\mfiedesc) and CFIE (\cfiedesc),
for the problem on the unit sphere with 4809 grid edges
as the wavenumber increases.
Close to $k=12.5$ there is an interior resonance, which causes the number of iterations for the EFIE and MFIE to explode, while the CFIE's iteration count 
stay small.}
  \label{fig:emcfie_k_vs_iterations}
\end{figure}

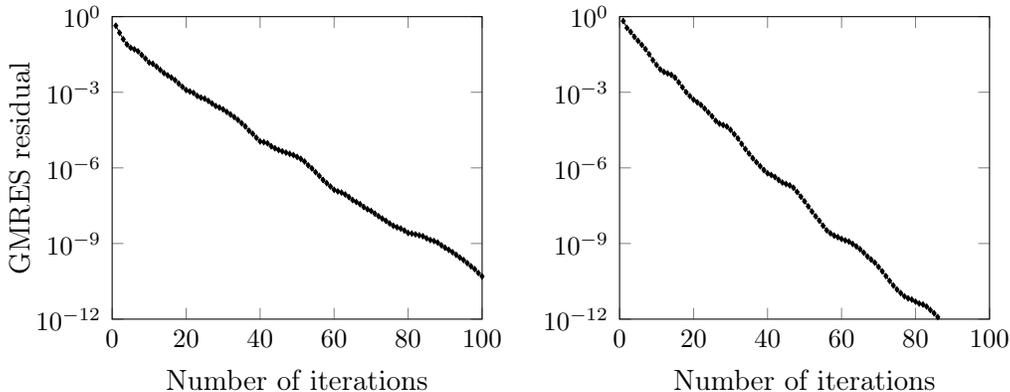
\begin{figure}
  \centering
\begin{tikzpicture}

\begin{axis}[small,
xlabel={Number of iterations},
ylabel={GMRES residual},
xmin=0, xmax=100,
ymin=1e-12, ymax=1,
ymode=log,
axis on top
]
\cfieplot[1]
table {%
1 0.4373063898947885
2 0.22817688488871662
3 0.12713784250870616
4 0.07915645530413551
5 0.05852135605931717
6 0.05065132234945462
7 0.04205042013097717
8 0.03017640639243531
9 0.021852543438948407
10 0.015162378504970559
11 0.013440151496548059
12 0.010427095038309433
13 0.007712243160785964
14 0.005789611228656972
15 0.0046888200536886315
16 0.003846218361350648
17 0.0030365343541674645
18 0.0021994966279654405
19 0.0016389727443546834
20 0.0012249163197671515
21 0.0010711561593516386
22 0.0009453334497711033
23 0.0007185719509574717
24 0.0006151225266529557
25 0.0005573534053956514
26 0.00045439782572086234
27 0.000360211905419642
28 0.00028479360735915985
29 0.0002471991444224034
30 0.0002102131178877988
31 0.00016597336161367946
32 0.00013122680518906816
33 0.00010378835973160902
34 8.021390035588735e-05
35 5.941061145474548e-05
36 4.420857162770492e-05
37 2.9927885135741852e-05
38 2.2617646617447514e-05
39 1.5678979520789337e-05
40 1.1221629083676806e-05
41 1.0522006443634307e-05
42 9.422818331206194e-06
43 7.237650648430926e-06
44 6.0216055337260255e-06
45 5.137723603049349e-06
46 4.521079788895738e-06
47 4.043479016761527e-06
48 3.608554841315347e-06
49 3.226214243991825e-06
50 2.725448498193884e-06
51 2.2789640323615222e-06
52 1.7933176339859556e-06
53 1.264199130580865e-06
54 9.55571745620248e-07
55 6.738763996595237e-07
56 4.810955328167548e-07
57 3.329384964051562e-07
58 2.492624395844363e-07
59 1.830249640594399e-07
60 1.3554909199591164e-07
61 1.1721515249580469e-07
62 1.0528039198173948e-07
63 8.810827776380833e-08
64 6.961721205680363e-08
65 5.293404779263348e-08
66 4.442477629656594e-08
67 3.647137674483354e-08
68 2.8174567949408635e-08
69 2.2908347096190068e-08
70 1.9588139785173388e-08
71 1.5360336608959664e-08
72 1.215059063072884e-08
73 9.804975680736766e-09
74 7.931984025310686e-09
75 6.259878309109986e-09
76 5.076574779881247e-09
77 4.4077355670760185e-09
78 3.940797492235414e-09
79 3.2877755135431335e-09
80 2.6255065133154436e-09
81 2.4878836369392526e-09
82 2.3121893071429596e-09
83 2.1184504675023303e-09
84 1.9477661212944846e-09
85 1.6088021797087694e-09
86 1.3783664960693275e-09
87 1.2589436766997765e-09
88 1.0895801138099029e-09
89 8.501796320230851e-10
90 6.821412185766986e-10
91 5.600156211535302e-10
92 4.540065113297074e-10
93 3.568892778289517e-10
94 2.779179711905906e-10
95 2.2143570997207834e-10
96 1.6664943155853298e-10
97 1.254141345284774e-10
98 9.668798124726964e-11
99 6.769966359368797e-11
100 4.969292111169209e-11
101 4.1329106738579936e-11
102 3.336650519092071e-11
103 2.6207364650223043e-11
104 1.8976976346698754e-11
105 1.4776571347882485e-11
106 1.271254898936919e-11
107 1.0205083866793797e-11
108 8.05690995855913e-12
109 6.607021496710239e-12
110 5.4496408614962384e-12
111 4.481161999698242e-12
112 3.383122831974794e-12
113 2.6000761440209493e-12
114 2.1672285968757617e-12
115 1.8415581992575907e-12
116 1.4599688196745181e-12
117 1.0305502555348721e-12
118 8.139332682650212e-13
119 7.526692682032942e-13
120 6.922212452346141e-13
121 6.301214120370001e-13
122 5.911626027863359e-13
123 4.672752328187876e-13
124 3.558656138287782e-13
125 3.245102028237358e-13
126 2.948981986833479e-13
127 2.5637692319762697e-13
128 2.1934718125017845e-13
129 1.710854218971463e-13
130 1.3681137154769596e-13
131 1.1155383680258454e-13
132 8.756514182002392e-14
133 6.571644049030959e-14
134 4.954763196441862e-14
135 3.913800495087723e-14
136 2.919432155080702e-14
137 2.234049864659408e-14
138 1.7883437097081658e-14
139 1.4970278091256567e-14
140 2.7967121852933968e-14
141 1.2661868308826386e-14
142 9.643751672179492e-15
143 8.038972782937625e-15
144 6.4790652479479525e-15
145 5.375832423418673e-15
146 4.331553323762209e-15
147 3.554392906176554e-15
148 3.0223542784534894e-15
149 2.6398885181927992e-15
150 2.1707829788719687e-15
151 1.6732653137771223e-15
152 1.2301010225975498e-15
153 9.376871125586264e-16
};
\end{axis}

\end{tikzpicture}
\begin{tikzpicture}

\begin{axis}[small,
xlabel={Number of iterations},
xmin=0, xmax=100,
ymin=1e-12, ymax=1,
ymode=log,
axis on top
]
\cfieplot[1]
table {%
1 0.6698350394166369
2 0.35691115454687605
3 0.2472389372136695
4 0.159046299036079
5 0.10946141345383167
6 0.07581185265778397
7 0.051563501663534496
8 0.03232426864029623
9 0.018794155038219705
10 0.012079077637457065
11 0.008099901937797473
12 0.006402530255188558
13 0.0056218798233696994
14 0.004901149660187099
15 0.00380626710150008
16 0.002483134665323827
17 0.0015825641097590305
18 0.001013477487045964
19 0.000706389378703392
20 0.0005088418436835736
21 0.0003995314578524577
22 0.0003146147304983913
23 0.00022482472892311343
24 0.0001584560842357429
25 0.00011185149880859154
26 7.360188022762342e-05
27 5.696448951039983e-05
28 5.006728288811138e-05
29 4.3556453108764784e-05
30 3.222516486236894e-05
31 2.1750032841651538e-05
32 1.468594005779081e-05
33 8.93237358742378e-06
34 5.611638488545726e-06
35 3.6887442813112964e-06
36 2.4153463146236065e-06
37 1.6822768455665847e-06
38 1.1691384505616782e-06
39 8.220495311766149e-07
40 6.112297477943904e-07
41 5.241485846212279e-07
42 4.3328150304107466e-07
43 3.3192432043131037e-07
44 2.625726879459657e-07
45 2.2904517485279644e-07
46 2.0260814677347916e-07
47 1.631299066071266e-07
48 1.1063719393179846e-07
49 7.213126078379678e-08
50 4.7376496176704876e-08
51 2.972225189968058e-08
52 1.857057556538907e-08
53 1.204379582530627e-08
54 8.091824770836364e-09
55 5.102675431745561e-09
56 3.3337002172375682e-09
57 2.5463239988593163e-09
58 2.069090843872446e-09
59 1.7885894860556132e-09
60 1.5199196436988307e-09
61 1.3342104431643455e-09
62 1.195978706476111e-09
63 9.635704303754571e-10
64 7.688008020000382e-10
65 5.95306739413687e-10
66 4.3389825365755985e-10
67 3.052381311194289e-10
68 2.3084208611041323e-10
69 1.739956997387548e-10
70 1.1798150160408821e-10
71 7.905329278550263e-11
72 5.1099390193376364e-11
73 3.294973075738334e-11
74 2.16021255440034e-11
75 1.4943187802815294e-11
76 1.0843187872737663e-11
77 8.045415588521422e-12
78 6.679218435977133e-12
79 5.9751960531745624e-12
80 4.936935968255188e-12
81 4.162853946396805e-12
82 3.816333129750898e-12
83 3.15782514020056e-12
84 2.3140041444212157e-12
85 1.7196714333938915e-12
86 1.215329538421827e-12
87 8.111432961782957e-13
88 5.448309740374523e-13
89 3.8447127509681225e-13
90 2.746014967013933e-13
91 2.1386046339190588e-13
92 1.5429569247859707e-13
93 1.0463357441882438e-13
94 7.127882265049431e-14
95 4.9111431096556473e-14
96 3.4180267685105856e-14
97 2.8317366100896118e-14
98 2.4584937163139962e-14
99 2.1045315126998064e-14
100 1.6882004220132357e-14
101 1.3715513561442482e-14
102 1.174058296414655e-14
103 1.0296358438433133e-14
104 8.507961674080308e-15
105 6.422113812719211e-15
106 4.64593333710567e-15
107 3.2847858219513652e-15
108 2.2986884635621903e-15
109 1.6196667772031942e-15
110 1.2576190750173876e-15
111 9.237320513654326e-16
};
\end{axis}

\end{tikzpicture}
  \caption{GMRES residuals vs iterations for the CFIE for the NASA almond (left) and Menger sponge (right) examples.}
  \label{fig:cfie_convergences}
\end{figure}

In Figure \ref{fig:emcfie_k_vs_iterations}, we demonstrate our CFIE implementation for a simple scattering problem on 
the unit sphere for varying wavenumbers. While the CFIE shows a bounded and only slowly growing number of GMRES iterations, the number of 
iterations for the EFIE and MFIE grow strongly in the neighborhood of an interior resonance close to $k=12.5$. In 
Figure \ref{fig:cfie_convergences}, we show the convergence of GMRES for the CFIE applied to the NASA almond and 
Menger sponge examples from Sections \ref{sec:efie} and \ref{sec:mfie}. 

Another construction of the CFIE based on the use of BC spaces was presented in 
\cite{Bagci2009}. In particular, the treatment of the MFIE component in that paper differs from the proposed formulation in this section.

\section{Concluding remarks}
\label{sec:conclusions}

\begin{figure}
  \centering
  \includegraphics[width=.5\textwidth]{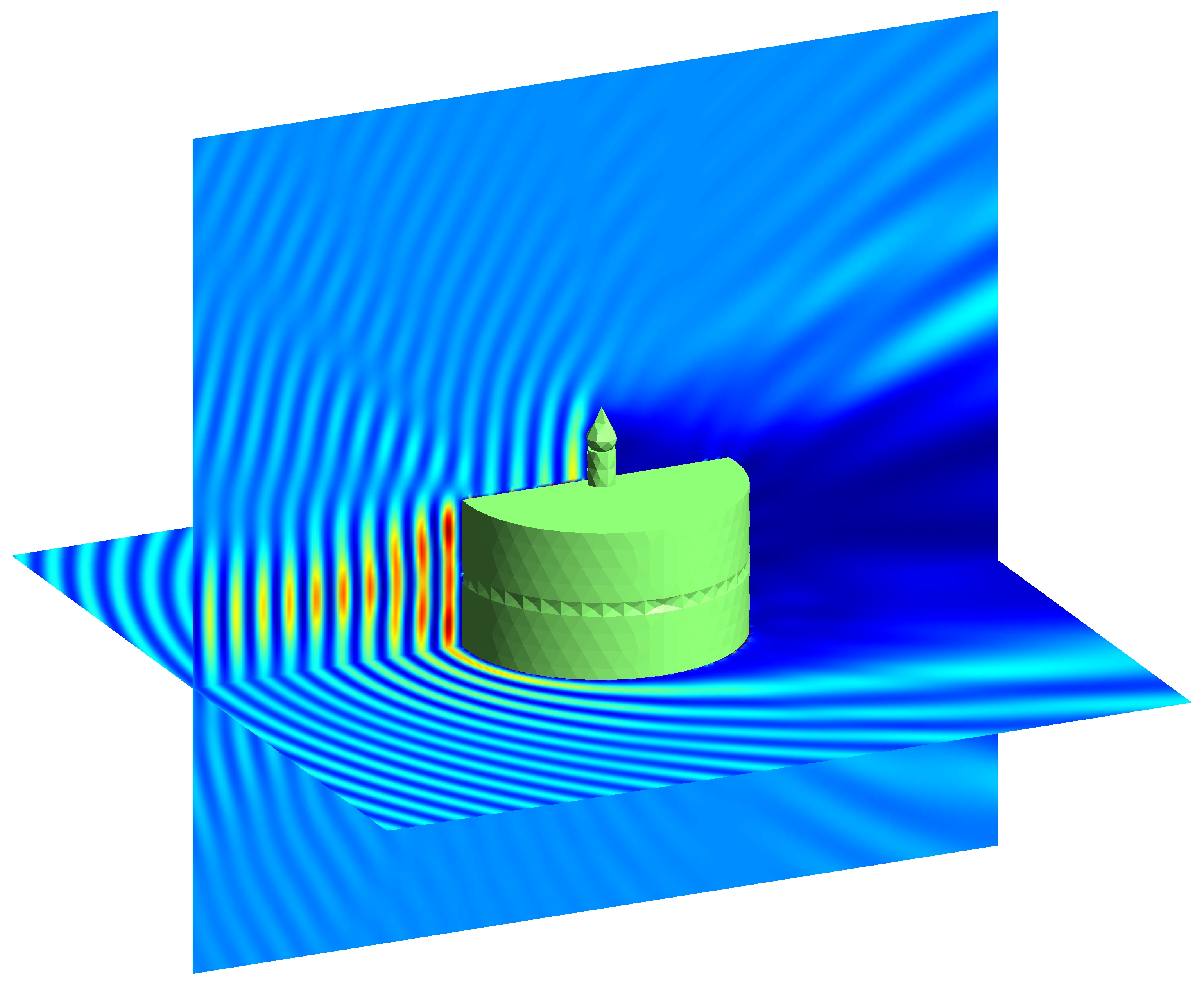}
  \caption{Slices of the squared electric field density of the wave $\bE\inc=[0,0,e^{\ii kx}]$, with $k=8$, scattering off a
birthday cake, computed using the preconditioned EFIE discretised on a grid with 3225 edges.}
  \label{fig:cake}
\end{figure}

We have only discussed the Maxwell solver aspects of BEM++. The goal is to make available advanced techniques for the solution of electromagnetic scattering problems while hiding complex details associated with robust boundary integral equation formulations for Maxwell. We have not discussed here the solution of Maxwell transmission problems. But the proposed framework fits this very well as a we need for each medium the corresponding multitrace operator $\bopA$. Work on single-trace and multi-trace transmission formulations is ongoing.

We have not included timing benchmarks in this paper since the focus is on the demonstration of the software framework and not the performance comparison of different solvers. BEM++ internally uses its own thread and MPI parallelised $\mathcal{H}$-matrix implementation for the discretisation of integral operators, which is very efficient for low-frequency problems, but also performs well in the medium frequency range.

Finally, we want to conclude this paper with an example adapted to the theme of this special issue. Figure \ref{fig:cake} shows the electromagnetic field around perfectly conducting birthday cake (not for eating!) computed with BEM++.

All examples in this paper were computed with version 3.1 of the BEM++ library, and will be made available in the form of IPython Notebooks on \url{www.bempp.org}.

\section*{Acknowledgements}
The authors would like to thank
Peter Monk for his continuing support of the BEM++ project. Indeed, he was the first external user of BEM++ when the project just started, and over the years he has constantly given helpful feedback and bug reports.

We would furthermore like to thank Dave Hewett (UCL) and Andrea Moiola (Reading) for helpful discussions regarding trace spaces of boundary integral operators for Maxwell.

The work of the second author was supported by Engineering and Physical Sciences Research Council grant EP/K03829X/1.

\bibliographystyle{elsarticle-num} 
\bibliography{references}

\end{document}